\documentclass[english, 12pt]{amsart}

\usepackage{geometry}
\usepackage[T1]{fontenc}
\usepackage[utf8]{inputenc}
\usepackage{babel}
\usepackage{amsmath}
\usepackage{amssymb}
\usepackage{amsthm}
\usepackage{textcomp}
\usepackage{enumitem}
\usepackage{dsfont}
\usepackage{mathrsfs}
\usepackage{cite}
\usepackage{color}
\usepackage[colorlinks=true,urlcolor=blue,linkcolor=blue]{hyperref}

\newcommand{\R}{\mathbb{R}}

\newcommand{\C}{\mathbb{C}}

\newcommand{\ii}{\mathrm{i}}
\newcommand{\dd}{\,\mathrm{d}}
\newcommand{\ee}{\mathrm{e}}

\newcommand{\eps}{\varepsilon}
\newcommand{\phe}{\varphi}
\renewcommand{\varepsilon}{\epsilon}

\newcommand{\la}{\langle}
\newcommand{\ra}{\rangle}
\def\build#1_#2^#3{\mathrel{
\mathop{\kern 0pt#1}\limits_{#2}^{#3}}}
\def\td_#1,#2{\mathrel{\mathop{\build\longrightarrow_{#1\rightarrow #2}^{}}}}
\def\dl_#1,#2{\mathrel{\mathop{\build=_{#1\rightarrow #2}^{}}}}
\usepackage{tikz}

\def\tdw_#1,#2{\mathrel{\mathop{\build\rightharpoonup_{#1\rightarrow #2}^{}}}}

\newtheorem{definition}{Definition}[section]
\newtheorem{theorem}[definition]{Theorem}
\newtheorem{lemma}[definition]{Lemma}
\newtheorem{proposition}[definition]{Proposition}
\newtheorem{corollary}[definition]{Corollary}
\newtheorem{remark}[definition]{Remark}
\numberwithin{equation}{section}

\title[A proof of the soliton resolution conjecture for the Benjamin--Ono equation]{A proof of the soliton resolution conjecture\\ for the Benjamin--Ono equation}
\date{\today}

\author[L.~Gassot]{Louise Gassot}
\address{CNRS and Department of Mathematics, University of Rennes, France}
\email{louise.gassot@cnrs.fr}

\author[P.~Gérard]{Patrick Gérard}
\address{Laboratoire de Mathématiques d'Orsay,  Université Paris-Saclay, Orsay, France} \email{patrick.gerard@universite-paris-saclay.fr}

\author[P.~D.~Miller]{Peter D. Miller}
\address{Department of Mathematics, University of Michigan, Ann Arbor, MI}
\email{millerpd@umich.edu}



\begin{document}

\maketitle

\begin{abstract}
We give a proof of the soliton resolution conjecture for the Benjamin--Ono equation, namely  every solution with sufficiently regular and decaying initial data can be written as a finite sum of soliton solutions with different velocities up to a radiative remainder term in the long--time asymptotics. We provide a  detailed correspondence between the spectral theory of the  Lax operator associated to the initial data and the different terms of the soliton resolution expansion.
The proof is based on a new use of a representation formula of the solution due to the second author, and on a detailed analysis of the distorted Fourier transform 
associated to the Lax operator.
\end{abstract}

\setcounter{tocdepth}{1}
\tableofcontents

\section{Introduction and main result}
This paper is devoted to the long-time behavior of solutions of the Benjamin--Ono equation,
\begin{equation}\label{BOinit}
\partial_tu-\partial_x|D|u+\partial_x(u^2)=0\ ,\ u(0,x)=u_0(x).
\end{equation}
Equation \eqref{BOinit} was introduced in 1967 \cite{Benjamin67} in order to model long, unidirectional internal gravity waves 
in a two-layer fluid with infinite depth, see e.g. the book \cite{KleinS21} and \cite{Paulsen24} for a recent derivation. Here $u=u(t,x)$ is a real--valued function and $|D|$ denotes the Fourier multiplier associated to the symbol $|\xi |$ acting on functions on the real line. Alternatively, $|D|=H\partial_x$, where $H$ denotes the Hilbert transform, normalized as
\begin{equation}\label{def:Hilbert}
Hf(x):=\frac{1}{\pi}\mathrm{p.v.}\int_\R \frac{f(y)}{x-y}\dd y.
\end{equation}
Global wellposedness theory of the initial--value problem \eqref{BOinit} in Sobolev spaces was the object of  intensive research from \cite{Saut79} to \cite{KillipLaurensVisan23-BO}, where it is shown that \eqref{BOinit} is solved by a continuous flow map on every Sobolev space $H^s(\R)$ for $s>-1/2$, so that, if $u_0\in H^s(\R)$, $u(t)$ is bounded in $H^s(\R)$ as $t\to \infty$. We refer to \cite{KillipLaurensVisan23-BO} for more references on this topic.  In this paper, we will be only interested in the case $s\geq 0$, see also \cite{MolinetP12}.

A very natural question is the description of the long--time dynamics of such solutions $u$. In order to be able to answer this question,
one must have in mind a class of special solutions of \eqref{BOinit}, which were the focus of Benjamin's original paper, namely solutions of the form
\begin{equation}\label{eq:solitonBO}
u(t,x)=R_p(x-c_pt),\ R_p(y):=\frac{2\mathrm{Im}p}{|y+p|^2},\ c_p:=\frac{1}{\mathrm{Im}p},
\end{equation}
where $p\in \C_+$ is a complex number with a positive imaginary part. A theorem by Amick and Toland \cite{AmickToland1991} establishes that all solutions of \eqref{BOinit} of the form $u_0(x-ct)$ are given by \eqref{eq:solitonBO}. Such solutions are called \emph{soliton solutions} of the Benjamin--Ono equation and are expected to be the building blocks of the whole dynamics. More precisely, the soliton resolution conjecture for this equation --- see e.g. \cite{IfrimTataru2019}--- roughly claims that every solution to \eqref{BOinit} behaves as $t\to \pm \infty $ like a  sum of such soliton solutions and a radiation term, namely a solution of the linearized equation of \eqref{BOinit} at $u=0$:
\begin{equation}\label{linearized}
\partial_tw =\partial_x|D|w,
\end{equation}
which we shall denote by $w(t):=\mathrm{e}^{t\partial_x|D|}w(0).$

 The purpose of this paper is to prove this conjecture for a large class of initial data $u_0$, including finite linear combinations of soliton solutions $R_p$ and, say, Schwartz functions. Our main result reads as follows.
\begin{theorem}\label{thm:solitonres}
Let $u_0\in H^1(\R)$ be a real--valued function such that $xu_0\in H^1(\R)$. Assume moreover that there exist $c_0\in \R$ and $v_0\in L^2(\R)$ such that
  \begin{equation}\label{hypu}
 \forall x\in \R,\  x^2u_0(x)=c_0+v_0(x).
 \end{equation}
 Then there exist a nonnegative integer $N$, elements $p_1,\dots ,p_N$ of $\C_+$, with $\mathrm{Im}(p_1)<\dots <\mathrm{Im}(p_N)$, and  real--valued functions $u_\infty ^\pm $ in $H^1(\R) $ such that
\begin{equation}\label{eq:solitonres}
 u(t,x)-\sum_{j=1}^N R_{p_j}(x-c_{p_j}t) -\mathrm{e}^{t\partial_x|D|}u_\infty^\pm  \longrightarrow 0 
 \end{equation}
in $H^1(\R) $ as $t\to \pm\infty$.
\end{theorem}
 Notice that the assumptions of Theorem \ref{thm:solitonres} on $u_0$ are  satisfied by any finite linear combination of  solitons, and more generally by any real--valued rational function $u_0\in L^2(\R)$ such that $xu_0\in L^2(\R)$, and also, of course, by any function in the Schwartz class. In the special case of a multi--soliton solution, namely if $u_0$ is a finite sum of functions $R_p$, Theorem \ref{thm:solitonres} was already known, see \cite{KaupM98}, \cite{Sun2020}, with no radiation term, $u_\infty^\pm=0$. On the other hand, Theorem \ref{thm:solitonres} was  more recently established in \cite{BlackstoneGGM24a}, in the special case $u_0=-R_p$, where $N=0$, so that the solution is purely radiative as $t\to \pm \infty$. A more precise formulation of Theorem \ref{thm:solitonres} is given in Section \ref{detail:solitonres} and the strategy of proof is presented in Section \ref{strategy}. 
 
Before coming to Section \ref{detail:solitonres}, let us mention related works dealing with soliton resolution for other Hamiltonian partial differential equations. 
The soliton resolution property seems to have been first observed in the case of the Korteweg--de Vries equation in \cite{EckSch83}, where the integrable structure is used through inverse scattering transform and the Gelfand--Levitan--Marchenko integral equation. Later, these asymptotics were revisited in \cite{DeiftVZ94}, using the Riemann--Hilbert formulation of the inverse scattering transform and the associated Deift--Zhou nonlinear steepest descent method. This approach was also used in \cite{CuccJenk16} to study the asymptotic stability of  multi--soliton solutions to the defocusing cubic nonlinear Schr\"odinger equation with non--zero boundary conditions at infinity, in \cite{BorgheseJMcL18} to study the long time asymptotics of the focusing cubic nonlinear Schr\"odinger equation, and in \cite{ChenLiu21} for the focusing modified Korteweg--de Vries equation. Finally, the case of the derivative nonlinear Schr\"odinger equation was studied in
 \cite{JenkLPS18} using again inverse scattering tools. Since then, similar Riemann--Hilbert/Deift--Zhou techniques have been applied to a myriad of other integrable dispersive equations. Notice that the Benjamin--Ono equation \eqref{BOinit} is also an integrable equation. However, its Cauchy problem on $\mathbb{R}$ lacks a justified inverse-scattering transform formulated as any kind of Riemann--Hilbert problem, making the large-time asymptotics inaccessible to the Deift--Zhou method.  Consequently, the proof of Theorem \ref{thm:solitonres} is very different from the above references, since inverse spectral theory is replaced by an explicit representation formula discovered by the second author in \cite{Gerard22}, see Theorem \ref{thm:explicit} below. Moreover, the upshot is also very different, since in Theorem \ref{thm:solitonres} the parameters $p_j$ are the same as $t\to +\infty$ and $t\to -\infty$, while this is not the case in the above references, where some non-trivial scattering map relates the soliton parameters as $t\to \pm \infty$. The difference also holds for the radiation term, since it is expressed as a free scattering term in Theorem \ref{thm:solitonres}, while it may be expressed as a modified scattering term in the above references.
 
 The soliton resolution property was recently established for some non--integrable Hamiltonian PDE, such as the energy critical focusing nonlinear wave equation with radial symmetry \cite{DuyKM23}, \cite{JenLaw23}, or the energy--critical wave map with equivariant symmetry \cite{DuyKM22}, \cite{JenLaw25}. Let us stress several differences of these two cases with the integrable ones. Firstly, the soliton solutions are no longer obtained by translation, but by dilation of the space variable. Secondly,  global existence does not always hold, hence the asymptotics hold either as $t\to \pm \infty$ or as $t$ tends to the blow--up time. Thirdly, the methods of proof are very different, using modulation and energy criticality, as well as energy channels for the wave equation. Such a modulation approach seems hopeless for the Benjamin--Ono equation, which is $H^{-1/2}$--critical, with no conservation law at this regularity level.
 
Let us conclude by mentioning an equation which may serve as a bridge between the  above two sets of equations, namely the Calogero--Moser derivative nonlinear Schr\"odinger equation \cite{GeLe-24}, \cite{KLV25}, which is mass critical and admits an integrable structure allowing an explicit formula very close to the Benjamin--Ono one. For this equation, a weak soliton resolution property was proved in \cite{KimK24} in the focusing case by means of modulation methods, while for the defocusing case scattering theory was very recently established in \cite{Chen25} using the integrable approach.

\subsection*{Acknowledgements}
L. Gassot was supported by the France 2030 framework program, the Centre Henri Lebesgue ANR-11-LABX-0020-01, and the ANR project HEAD--ANR-24-CE40-3260.\\
P. G\'erard was partially supported by the French Agence Nationale de la Recherche under the ANR project ISAAC--ANR-23--CE40-0015-01.\\
P. D. Miller was partially supported by the National Science Foundation under grants DMS-2204896 and DMS-2508694.  \\

\section{A more precise formulation of the main result}\label{detail:solitonres}
\subsection{The Hardy space and the Lax operator}
Let us now introduce the spectral data that will be relevant parameters in the long-time asymptotics stated in Theorem \ref{thm:solitonres}. We will use  the integrable structure for the Benjamin--Ono equation, for which we refer e.g. to \cite{Gerard26}.
We denote by $L^2_+(\R)$ the set of functions $f\in L^2(\R)$ whose Fourier transform $\widehat{f}$ is supported in the half-line $[0,+\infty)$. An element $f$ in the Hardy space $L^2_+(\R)$ extends as a  holomorphic function on the upper half-plane $\C_+$ satisfying
\[
\sup_{y>0}
\int_{\R}|f(x+\ii y)|^2\dd x<+\infty,
\]
in which case $f$ can be identified with the limit in $L^2(\R)$ of $f(\cdot+\ii y)$ as $y\to 0^+$.
We denote by $\Pi$ the orthogonal projector from $L^2(\R)$ onto $L^2_+(\R)$. The Toeplitz operator $T_{b}$ associated to $b\in L^{\infty}(\R)$ is then given by the formula
\[
\forall f\in L^2_+(\R), \quad
T_{b}f:=\Pi (b f).
\]
The Lax operator for the Benjamin-Ono equation is given by $L_{u_0}=D-T_{u_0}$, where $D:=-\ii\frac{\dd }{\dd x}$, with domain $D(L_{u_0})=H^1_+(\R):=H^1(\R)\cap L^2_+(\R)$.

Finally, we define the operator $X^*$ by its domain
\[
D(X^*)
	=\{f\in L^2(\R)\colon \widehat{f}|_{(0,\infty)}\in H^1(0,\infty)\},
\]
and the formula
\[
\widehat{X^*f}(\xi)
	:=\ii \frac{\dd }{\dd \xi}\widehat{f}(\xi),
\quad
\xi>0.
\]
Applying the inverse Fourier transform, we find
\[
X^*f(x)=xf(x)+\frac{I_+(f)}{2\ii\pi},
\]
where $I_+(f):=\widehat{f}(0^+)$. The operator $X^*$ can also be seen as the adjoint on $L^2_+(\R)$ of  multiplication by $x$, see \cite{Gerard26}.

Next we define  soliton parameters. We denote by $\lambda_1,\dots, \lambda_N$ the --- simple and negative --- eigenvalues of the Lax operator $L_{u_0}=D-T_{u_0}$ on the Hardy space $L^2_+(\R)$, and by $\phe_1,\dots, \phe_N$ the corresponding eigenfunctions, normalized in $L^2$. The finiteness of $N$ has been established in \cite[Theorem 1.2]{Wu16a}. For every $j\in \{1,\dots, N\}$, we set
\begin{equation}\label{def:parameter}
 p_j=-\langle X^*\phe_j,\phe_j\rangle .
 \end{equation}

Recall that, according to Wu's identities~\cite[Lemma 2.5 and (2.23)]{Wu16a} --- see also \cite{BadreddineKV25} if $u$ is less  regular---,
\begin{equation}\label{eq:Wu} |\langle \phe_j, \Pi u_0\rangle |^2=-2\pi \lambda_j\ ,\ \lambda_j I_+(\phe_j)=-\langle \phe_j,\Pi u_0\rangle ,
\end{equation}
and consequently, expressing $\langle X^*\phe_j,\phe_j\rangle$ on the Fourier side and integrating by parts, we find
\begin{equation}\label{Imp}
\mathrm{Im}\, p_j=\frac{\vert I_+(\phe_j)\vert ^2}{4\pi}=\frac{1}{2|\lambda_j|}>0.
\end{equation}

Now we define generalized eigenfunctions and radiation profiles.
For every $\lambda>0$, we denote by $m_\pm=m_\pm(x,\lambda)$ the unique solution in the space $L^\infty(\R)$  of the equation
\begin{equation}\label{def:m}
 (L_{u_0}-\lambda\,\mathrm{Id})m_\pm=0\ ,\ \lim_{x\to \pm\infty}{\mathrm{e}^{-\ii\lambda x}m_\pm(x,\lambda)}=1.
 \end{equation}
For every $f\in L^2_+(\R)$, we denote by $\widetilde f^\pm $ the distorted Fourier transform of $f$ defined by 
\begin{equation}\label{distortedFT}
 \widetilde f^{\pm} (\lambda)=\int_\R f(x)\overline {m_\pm(x,\lambda)}\dd x.
 \end{equation}
Finally, we define  real--valued functions $u_\infty^\pm \in L^2(\R)$ by
\begin{equation}\label{def:radiationprofile}
\forall \lambda >0\ ,\ \widehat u_\infty^\pm (\lambda )=\widetilde{\Pi u_0}^\mp(\lambda).
\end{equation}
We recall that the existence and uniqueness of $m_-$ was proven in~\cite[Theorem 26]{BlackstoneGGM24a}; the analogous result for $m_+$ can be proved similarly. Moreover, the following distorted Plancherel identity was proven in~\cite[Theorem 26]{BlackstoneGGM24a}:
\begin{equation}\label{Planchereldist}
\sum_{j=1}^N \vert \langle f,\phe_j\rangle \vert ^2+\int_0^\infty \vert \widetilde f^\pm(\lambda)\vert ^2 \frac{\dd\lambda}{2\pi} =\int_\R |f(x)|^2\dd x.
\end{equation}

\subsection{The main result in more detail}
The following result relates the parameters of Theorem \ref{thm:solitonres} to the spectral quantities introduced above.
\begin{theorem}[Soliton resolution and spectral data]\label{thm:main}
Under the assumptions of Theorem \ref{thm:solitonres}, the function $r^\pm$ defined by
\[ r^\pm (t,x):=u(t,x)-\sum_{j=1}^N R_{p_j}(x-c_{p_j}t)-\mathrm{e}^{t\partial_x|D_x|}u_\infty^\pm (x)\]
satisfies
\begin{equation}\label{cancelr}
\Vert r^\pm (t,\diamond)\Vert_{H^1}^2:=\int_\R |r^\pm (t,x)|^2\dd x +\int_\R |\partial_xr^\pm (t,x)|^2\dd x\td_t,{\pm \infty}0.
\end{equation}
\end{theorem}
\begin{corollary}\label{cor:special}
Under the assumptions of Theorem \ref{thm:solitonres}, 
\begin{itemize}
\item Purely radiative asymptotics --- $N=0$ --- holds if and only if the operator $L_{u_0}$ is positive. This is in particular the case if $u_0\leq 0$.
\item Given a positive integer $N$, a $N$--soliton asymptotics holds for $t\to +\infty $ --- $u_\infty ^+=0$--- or for $t\to -\infty$ --- $u_\infty ^-=0$--- if and only if $u_0$ is a sum of $N$ soliton profiles, namely there exist  $\tilde p_1,\dots, \tilde p_N\in \C_+$ such that
\[ u_0(x)=\sum_{j=1}^N R_{\tilde p_j}(x).\]
\end{itemize}
\end{corollary}
Let us briefly indicate how to deduce Corollary \ref{cor:special} from Theorem \ref{thm:main}. The first statement is a consequence of the already--quoted results of \cite{Wu16a}, namely that, if $u_0\in L^\infty \cap L^2$ and $xu_0\in L^2$,  $L_{u_0}$ has no eigenvalue if and only if it is a positive operator. The particular case $u_0\leq 0$ is then obvious.
As for the second statement, we observe that the distorted Plancherel  formula \eqref{Planchereldist} applied to $f=\Pi u_0$ implies that $u_\infty ^\pm=0$ if and only if $\Pi u_0$ is a linear combination of eigenfunctions of $L_{u_0}$. Then the conclusion is a direct consequence of Theorem 4.8 of \cite{Sun2020}, where Theorem \ref{thm:main} is also proved in the particular case of multi--solitons.
\begin{remark}\
Notice that, with the notation of ~\cite{BlackstoneGGM24a}, see in particular Remark 15, we have $ \hat u_\infty ^+(\lambda)=\alpha (\lambda)$ and $ \hat u_\infty ^-(\lambda)=-\ii \beta (\lambda)$. In  view  of ~\cite[Proposition 30]{BlackstoneGGM24a}, we observe that there exists a function $\ell (\lambda )=\ell_+(m_-,\lambda)$ of modulus $1$ such that 
\begin{equation}\label{scattering}
 \forall \lambda >0, 
 \quad
 \hat u_\infty ^+(\lambda)=\ell (\lambda) \hat u_\infty ^-(\lambda),
 \end{equation}
so that the scattering operator is a unitary Fourier multiplier.
\end{remark}

\section{Strategy of proof}\label{strategy}
In this section, we sketch the proof of Theorem~\ref{thm:main}. We shall make use of the explicit formula for $u$ from~\cite{Gerard22}.

\begin{theorem}[Explicit formula\cite{Gerard22}]\label{thm:explicit}
The solution $u\in\mathcal{C}(\R,H^1(\R))$ of the Benjamin-Ono equation~\eqref{BOinit} with initial datum $u(0)=u_0\in H^1(\R)$ is given by $u(t,x)=\Pi u(t,x)+\overline{\Pi u(t,x)}$, with
\begin{equation}\label{explicit}
\forall t\in \R,\  \forall z\in \C_+,\quad
  \Pi u(t,z)=\frac{1}{2\ii\pi}I_+(f_{t,z}),\quad f_{t,z}:=(X^*-2tL_{u_0}-z\,\mathrm{Id})^{-1}\Pi u_0.
\end{equation}
\end{theorem}
\begin{remark}\label{invertibility}
Note that the existence of $(X^*-2tL_{u_0}-z\,\mathrm{Id})^{-1}$ for every $z\in \C_+$ comes from the maximal dissipativity of the operator $X^*-2tL_{u_0}$, see the end of Section 3 of \cite{Gerard22}. A priori the vector $f_{t,z}$ belongs to the domain of $X^*-2tL_{u_0}$. We will see in Section \ref{sec:estimates} that, due to the properties of $\Pi u_0$, it belongs to the domain of $L_{u_0}$ and to the domain of $X^*$.
\end{remark}
\subsection{Reduction to the case $t\to +\infty$}
We observe that, if $u$ is the solution of \eqref{BOinit}, and if $u^* (t,x):=u(-t,-x)$, then $u^*$ still satisfies the Benjamin--Ono equation with initial datum $u^* _0(x):=u_0(-x)$. Notice that, since $u_0$ is real valued, $\hat u^*_0(\xi )=\overline{\hat u_0(\xi)}$ for every $\xi \in \R$. As a consequence, if $\phe_j$ is an eigenfunction of $L_{u_0}$ associated to the eigenvalue $\lambda_j$, the function $\phe_j^* \in H^1_+(\R)$ defined by 
\[ \widehat {\phe_j^*}(\xi)=\overline{\phe_j(\xi)}\]
is an eigenfunction of $L_{u_0^*}$ with the same eigenvalue, and 
\[ \la X^*\phe_j^* ,\phe_j^*\ra =-\overline {\la X^*\phe_j,\phe_j\ra}.\]
Similarly, if the generalized eigenfunctions $m_\pm $ are associated to $u_0$ by Definition \ref{def:m}, then the  generalized eigenfunctions  $m^*_\pm $ associated to $u_0^*$ are given by
\[ m_\pm ^* (x,\lambda)=\overline{m_\mp (-x,\lambda)}.\]
Consequently, the scattering profiles $u_\infty ^{*,\pm} $ associated to $u_0^*$ by Definition \ref{def:m} are given by 
\[ \forall \lambda >0, 
\quad
\hat u_\infty^{*,\pm}(\lambda)=\int_\R \Pi u_0^* (x)\overline {m_\mp^* (x,\lambda)}\dd x=\int_\R \overline{\Pi u_0 (-x)}m_\pm (-x,\lambda)\dd x
=\overline{\hat u_\infty ^\mp (\lambda)},\]
or equivalently $u_\infty^{ *,\pm}(x)=u_\infty^\mp(-x)$. Hence the statement of Theorem \ref{thm:main} for $t\to +\infty$ and $u^*$ is exactly the statement of Theorem \ref{thm:main} for $t\to -\infty$ and $u$. \\
Therefore we may restrict ourselves to the case $t\to +\infty$, which we shall do from now on, and we will simplify the notation by setting \[\tilde f(\lambda):=\tilde f^-(\lambda),\  u_\infty :=u_\infty ^+,\  r:=r^\pm. \]

\subsection{Sufficient conditions based on weak convergence}

 By a functional--analytic argument using the distorted Plancherel formula \eqref{Planchereldist}, we are reduced to establishing the following two limits.
 
\begin{lemma}[Sufficient conditions for Theorem~\ref{thm:main}]
Theorem~\ref{thm:main} holds if  the following two properties are satisfied.
\begin{itemize}
\item Soliton limit. For every $j\in \{1,\dots, N\}$, for every $z\in \C_+$,
\begin{equation}\label{soliton} \Pi u(t,z-2t\lambda_j)\td_t,\infty \frac{\ii}{z-\langle X^*\phe_j,\phe_j\rangle }. \end{equation}
\item Radiation limit. For every $z\in \C_+$ with $|z|$ small enough, 
\begin{equation}\label{scat} (2t)^{1/2}\mathrm{e}^{\ii t\lambda^2}\Pi u(t,z-2t\lambda)
	\tdw_t,{+\infty} \frac{\mathrm{e}^{\ii\pi/4}}{\sqrt{2\pi}}\mathrm{e}^{\ii\lambda z}\widetilde{\Pi u_0}({\lambda}), \end{equation}
in the sense of weak convergence in $L^2(0,+\infty )$ in the $\lambda$ variable.
\end{itemize}
\end{lemma}
\begin{proof} {\bf First step. Reformulation of the radiation limit.} Assume  the radiation limit \eqref{scat} holds for $\mathrm{Im}z$  small enough. Then we claim that it also holds for $z=0$. Indeed, notice that, since $\Pi u(t,.)$ is bounded in $L^2$, an elementary change of variable implies that, for every $z\in \C_+$, 
\begin{equation}\label{def:Ftz}
 F_{t,z}(\lambda):=(2t)^{1/2}\mathrm{e}^{\ii t\lambda^2}\Pi u(t,z-2t\lambda)
 \end{equation} is bounded in $L^2(\R)$. Moreover, since $\partial_z\Pi u(z)=\Pi (\partial_xu)(z)$, the $H^1$ conservation law implies that $\partial_zF_{t,z}(\lambda)$ is also bounded in $L^2(\R)$. Assume that \eqref{scat} holds for $0<\mathrm{Im}z\leq \delta $, fix $\phe $ a test function on $(0,\infty )$, and set, for $y\in (0,\delta ]$,
\begin{equation}\label{def:gt}
g_t(y):=\int_0^\infty F_{t,\ii y}(\lambda)\phe (\lambda)\dd \lambda ,\  g_{\infty} (y):=\frac{\mathrm{e}^{\ii\pi/4}}{\sqrt{2\pi}}\int_0^\infty \mathrm{e}^{-\lambda y}\widetilde{\Pi u_0}({\lambda})\phe (\lambda)\dd\lambda.
\end{equation}
Then we have $g_t(y)\to g_{\infty} (y)$ as $t\to +\infty $ for every $y\in (0,\delta ]$, and $g'_t(y)$ is uniformly bounded on $(0,\delta ]$ as $t\to \infty$. Consequently,
the convergence of $g_t$ to $g_{\infty}$ is uniform, thus $g_t(0^+)\to g_{ \infty} (0^+)$ as $t\to + \infty$, which, since $\phe $ is arbitrary, exactly means \eqref{scat} for $z=0$, namely
\begin{equation}\label{scat2}
(2t)^{1/2}\mathrm{e}^{\ii t\lambda^2}\Pi u(t,-2t\lambda)
	\tdw_t,{+\infty} \frac{\mathrm{e}^{\ii\pi/4}}{\sqrt{2\pi}}\widetilde{\Pi u_0}({\lambda}), \end{equation}
in the sense of weak convergence in $L^2(0,+\infty )$ in the $\lambda$ variable. At this stage we can reformulate \eqref{scat2} as
\begin{equation}\label{scat3}
\ee^{it\partial_x^2}\Pi u (t,\diamond)\tdw_t,{+\infty}\Pi u_\infty.
\end{equation}
Indeed, it is well known that, for every $f\in L^2(\R)$,
\[ \ee^{-\ii t\partial_x^2}f=\ee^{-\ii \frac{x^2}{4t}}\frac{\ee^{ \ii \pi/4}}{\sqrt{4\pi t}}\hat f\left (\frac{-x}{2t}\right )+o(1),\]
as $t\to + \infty$, in $L^2(\R)$, so that, for every $f\in L^2_+(\R)$, 
\begin{align*}
\la  \ee^{\ii t\partial_x^2}\Pi u(t,\diamond),f\ra &= (2t)^{1/2}\frac{\ee^{- \ii \pi/4}}{\sqrt{2\pi }}\int_0^\infty \Pi u(t,-2t\lambda)\ee^{\ii t\lambda^2}\overline{\hat f(\lambda)}\dd\lambda +o(1)\\
&\td_t,{+\infty} \frac{1}{2\pi} \int_0^\infty \hat u_\infty (\lambda)\overline{\hat f(\lambda)}\dd\lambda =\la \Pi u_\infty, f\ra , 
\end{align*}
in view of \eqref{scat2}.
\begin{remark}\label{Hsequicontinuity}
Notice that the above step can be made under a weaker assumption on $u_0$, say $u_0\in H^s$ for some $s>0$. Indeed, in that case the conservation laws imply that $\Pi u(t,\diamond)$ is uniformly bounded in $H^s(\R )$, so that, if $0<y<y'\leq \delta$,
\[ \int_\R |F_{t,\ii y}(\lambda)-F_{t,\ii y'}(\lambda)|^2\dd\lambda \lesssim \int_0^\infty (\mathrm{e}^{-y\xi }-\mathrm{e}^{-y'\xi})^2|\hat u(t,\xi )|^2\dd\xi \lesssim (y'-y)^{2s}.\]
Hence the family $g_t$ defined in \eqref{def:gt} is uniformly H\"older continuous, and $g_t(0^+)\to g_\infty (0^+),$ whence \eqref{scat2}.
\end{remark}
{\bf Second step. Strong convergence in $L^2$.} Next, if $r=r^+$ is the function defined in Theorem \ref{thm:main},  we observe that
\[ \Pi r(t,x)=\Pi u(t,x)-\sum_{j=1}^N q_j(x+2\lambda_jt)-\ee^{-\ii t\partial_x^2}\Pi u_\infty (x),\ q_j(x):=\frac{\ii}{x-\la X^*\phe_j,\phe_j\ra }.\]
Let us expand the quantity $\Vert \Pi r(t,\diamond )\Vert_{L^2}^2=\la \Pi r(t,\diamond ), \Pi r(t,\diamond )\ra $. 

Since $\lambda _j\ne \lambda_k$ of $j\ne k$, it is easy to check that
\begin{equation}\label{qjk}
\int_\R q_j(x+2\lambda_jt)\overline q_k(x+2\lambda_kt)\dd x \td_t,{+\infty} 0.
\end{equation}
Furthermore, since $q_j\in L^p$ for every $p>1$, the dispersion inequality for the Schr\"odinger group  implies that $\ee ^{\ii t\partial_x^2}q_j$ tends to $0$ in $L^q$ for every $q<\infty $, hence 
$\ee ^{\ii t\partial_x^2}q_j(\diamond +2\lambda_jt )$ tends weakly to $0$ in $L^2(\R)$. Consequently, 
\begin{equation}\label{qr}
\la q_j(\diamond +2\lambda_jt),\ee^{-\ii t\partial_x^2}\Pi u_\infty (x) \ra  \td_t,{+\infty} 0.
\end{equation}
Next we look at the inner products of $\Pi u(t,\diamond )$ with the other terms. We have
\begin{equation}\label{uq}
\la \Pi u(t,\diamond), q_j(\diamond +2\lambda_jt)\ra =\la \Pi u(t,\diamond -2\lambda_jt) ,q_j\ra \td_t,{+\infty} \la q_j,q_j\ra =2\pi |\lambda_j| ,
\end{equation}
in view of \eqref{soliton} and of \eqref{Imp}.
Furthermore, in view of \eqref{scat3}, we have 
 \begin{equation}\label{uuinfty}
\la \Pi u(t,\diamond), \ee^{-\ii t\partial_x^2}\Pi u_\infty\ra   \td_t,{+\infty}\frac{1}{2\pi}\Vert \widetilde {\Pi u_0}\Vert _{L^2}^2.
\end{equation}
Summing up, in view of \eqref{qjk}, \eqref{qr}, \eqref{uq}, \eqref{uuinfty}, and of the $L^2$ conservation law for the Benjamin--Ono evolution,
we obtain, as $t\to +\infty $,
\begin{align*}
\Vert \Pi r(t,\diamond)\Vert_{L^2}^2&=\Vert \Pi u(t,\diamond)\Vert_{L^2}^2-\sum_{j=1}^N 2\pi |\lambda_j|-\frac{1}{2\pi}\Vert \widetilde {\Pi u_0}\Vert _{L^2}^2+o(1)\\
&=\Vert \Pi u_0\Vert_{L^2}^2-\sum_{j=1}^N 2\pi |\lambda_j|-\frac{1}{2\pi}\Vert \widetilde {\Pi u_0}\Vert _{L^2}^2+o(1)\\
&=o(1),
\end{align*}
where we have applied \eqref{Planchereldist} to $f=\Pi u_0$ and \eqref{eq:Wu} in the last identity. Since $\Vert r(t,\diamond)\Vert_{L^2}^2=2\Vert \Pi r(t,\diamond)\Vert_{L^2}^2, $ we have proved
\begin{equation}\label{cancelrfirst}
 \Vert r(t,\diamond)\Vert_{L^2}^2\td_t,{+\infty} 0.
 \end{equation}
{\bf Third step. Strong convergence in $H^1$.} Next we come to the proof of 
\[ \Vert \partial_x r(t,\diamond)\Vert_{L^2}^2\td_t,{+\infty} 0,\]
which we shall establish equivalently  as 
\begin{equation}\label{eq:cancelrprime}
\Vert \partial_x \Pi r(t,\diamond)\Vert_{L^2}^2\td_t,{+\infty} 0.
\end{equation}
We shall appeal to the conservation law
\begin{equation}\label{eq:conslawH1}
\Vert  L_{u(t,\diamond)}\Pi u(t,\diamond)\Vert_{L^2}^2=\Vert L_{u_0}\Pi u_0\Vert_{L^2}^2,
\end{equation}
which implies that $u(t,\diamond )$ is bounded in $H^1(\R)$. Furthermore, since 
\[ \widetilde{L_{u_0}\Pi u_0}(\lambda )=\lambda \widetilde{\Pi u_0}(\lambda )=\lambda \hat u_\infty (\lambda ), \]
we infer that $u_\infty \in H^1(\R)$, so that $r(t,\diamond )$ and $\Pi r(t,\diamond )$ are bounded in $H^1(\R)$. In view of \eqref{cancelrfirst}, we infer by interpolation that $r(t,\diamond )\to 0$ in $H^s(\R)$ strongly for every $s<1$, and that $r(t,\diamond )\rightharpoonup 0$ weakly in $H^1(\R)$.

Then we proceed with $\partial_xr(t,\diamond )$ as we did for $r(t,\diamond)$ in the previous part of the proof. Using the same chain of arguments, one gets, as $t\to +\infty$,
\begin{equation}\label{expandrprime}
\Vert \partial_x\Pi r(t,\diamond)\Vert_{L^2}^2=\Vert \partial_x\Pi u(t,\diamond)\Vert_{L^2}^2-\sum_{j=1}^N \Vert \partial_xq_j\Vert_{L^2}^2-\Vert \partial_x\Pi u_\infty\Vert_{L^2}^2+o(1).
\end{equation}
Next we transform the right hand side of \eqref{expandrprime} as follows. Since $\Pi $ is bounded on $L^p$ for every $p\in (1,\infty)$ \cite{Stein93}, the bilinear map
\[ B(v,w)=T_v\Pi w\]
is  continuous from $L^4(\R)\times L^4(\R)$ into $L^2(\R)$. On the other hand, by the dispersion inequality applied to the free Schr\"odinger group, we know 
that, for every $v\in H^1(\R)$, for every $p\in (2,\infty ]$,
\begin{equation}\label{dispLp}
 \Vert \mathrm{e}^{-it\partial_x^2}v \Vert_{L^p}\td_t,\infty 0.
 \end{equation}
Indeed, since $\mathrm{e}^{-it\partial_x^2}$ is an isometry of $H^1$, it is uniformly bounded from $H^1$ to $L^p$, and it suffices to check \eqref{dispLp} for $v$ in a dense subspace of $H^1$, say $H^1\cap L^1$, which is trivial since  $\mathrm{e}^{-it\partial_x^2}$ sends $L^1$ into $L^\infty$ with a norm $O(t^{-1/2})$. Consequently, if we set
\[ \tilde r(t,x):=r(t,x)+\mathrm{e}^{t\partial_xD}u_\infty (x),\]
we have $\Vert \tilde r(t,\diamond )\Vert _{L^p}\to 0$ and hence, setting
\[ R(t,x):=\sum_{j=1}^N R_{p_j}(x+2\lambda_jt)=u(t,x)-\tilde r(t,x),\]
we have \[ B(u(t,\diamond ), u(t,\diamond ))=B(R(t,\diamond ),R(t,\diamond ))+o(1),\]
in $L^2$ as $t\to +\infty$. Hence 
\begin{align*}
\Vert L_{u(t,\diamond)}\Pi u(t,\diamond)\Vert_{L^2}^2=&\Vert D\Pi u(t,\diamond )-B(u(t,\diamond ),u(t,\diamond))\Vert_{L^2}^2\\
=&\Vert D\Pi u(t,\diamond )-B(R(t,\diamond ),R(t,\diamond))\Vert_{L^2}^2+o(1)\\
=&\Vert D\Pi u(t,\diamond )\Vert_{L^2}^2-2\mathrm{Re}\la D\Pi u(t,\diamond ), B(R(t,\diamond ),R(t,\diamond))\ra +\\
&+\Vert B(R(t,\diamond ),R(t,\diamond))\Vert_{L^2}^2+o(1).
\end{align*}
Next we observe that
\begin{align*} \la D\Pi u(t,\diamond ), B(R(t,\diamond ),R(t,\diamond))\ra =&\la D\Pi R(t,\diamond ), B(R(t,\diamond ),R(t,\diamond))\ra +\la \Pi \tilde r(t,\diamond ), DB(R(t,\diamond ),R(t,\diamond))\ra \\
=&\la D\Pi R(t,\diamond ), B(R(t,\diamond ),R(t,\diamond))\ra +o(1),
\end{align*}
because $DB(R(t),R(t))$ is bounded in $L^q(\R)$ for every $q>1$ on the one hand, and because $\Pi \tilde r(t,\diamond )$ tends to $0$ in $L^p$ for every $p>2$ on the other hand.  This leads to 
\begin{align*}
\Vert L_{u(t,\diamond)}\Pi u(t,\diamond)\Vert_{L^2}^2
	=&\Vert D\Pi u(t,\diamond )\Vert_{L^2}^2-2\mathrm{Re}\la D\Pi R(t,\diamond ), B(R(t,\diamond ),R(t,\diamond))\ra +\\
&+\Vert B(R(t,\diamond ),R(t,\diamond))\Vert_{L^2}^2+o(1)\\
	=&\Vert \partial_x\Pi u(t,\diamond )\Vert_{L^2}^2+\sum_{j=1}^N[-2\mathrm{Re}\la D\Pi R_{p_j}, B(R_{p_j},R_{p_j})\ra +\Vert B(R_{p_j},R_{p_j})\Vert_{L^2}^2]+o(1),
\end{align*}
 since the pairwise interactions of the terms $R_{p_j}(\diamond+2\lambda_jt)$ cancel at the first order in view of  $\lambda_j\ne \lambda_k$.
  
  In the latter identity, let us replace $\Vert \partial_x\Pi u(t,\diamond )\Vert_{L^2}^2$ by its expression provided by \eqref{expandrprime}. We obtain
 \begin{align*}
\Vert L_{u(t,\diamond)}\Pi u(t,\diamond)\Vert_{L^2}^2&=\sum_{j=1}^N\Vert L_{R_{p_j}}q_j\Vert_{L^2}^2+\Vert \partial_x\Pi u_\infty\Vert_{L^2}^2+\Vert \partial_x\Pi r(t,\diamond)\Vert_{L^2}^2+o(1)\\
&=\sum_{j=1}^N2\pi |\lambda_j|^3+\Vert \partial_x\Pi u_\infty\Vert_{L^2}^2+\Vert \partial_x\Pi r(t,\diamond)\Vert_{L^2}^2+o(1),
\end{align*}
because $L_{R_{p_j}}q_j=\lambda_j q_j$ and $\Vert q_j\Vert_{L^2}^2=2\pi |\lambda_j|.$ Applying the distorted Plancherel formula \eqref{Planchereldist} to $f=L_{u_0}\Pi u_0$, we finally infer 
\begin{equation}\label{energybalance}
\Vert L_{u(t,\diamond)}\Pi u(t,\diamond)\Vert_{L^2}^2=\Vert L_{u_0}\Pi u_0\Vert_{L^2}^2+\Vert \partial_x \Pi r(t,\diamond)\Vert_{L^2}^2+o(1),
\end{equation}
which, in view of \eqref{eq:conslawH1}, leads to \eqref{eq:cancelrprime} and completes the proof.
\end{proof}
\subsection{Soliton limit}
In Section \ref{sec:solitonlimit},  will obtain the soliton limit \eqref{soliton} by a direct application of the explicit formula \eqref{explicit}  combined with considerations on the action of the operator $X^*$ on the eigenfunctions of $L_{u_0}$.
\subsection{Radiation limit}
In Sections \ref{sec:radiationlimit} - 
\ref{3.9}, the radiation limit will also be obtained through the explicit formula \eqref{explicit}, but the proof will be much more intricate. An important ingredient will be that the $z$-translate 
\[ f_{t,z-2t\lambda}=(X^*-2t(L_{u_0}-\lambda\,\mathrm{Id})-z\,\mathrm{Id})^{-1}\Pi u_0\]
of the function $f_{t,z}$ arising in \eqref{explicit} is such that $(x,\lambda)\mapsto\sqrt {t} f_{t,z-2t\lambda}(x)$ converges weakly to $0$ in $L^2(\R \times \Lambda )$ for any segment $\Lambda \subset (0,\infty )$
(see Section \ref{sec:estimates}). Our proof will use properties of the distorted Fourier transform defined in \eqref{distortedFT}, which we establish in Appendix \ref{sec:distorted}.

\section{The soliton limit}\label{sec:solitonlimit}
In this section, we prove property \eqref{soliton}. Firstly, we give more information on the structure of formula \eqref{explicit}.
\begin{lemma}\label{L2boundexplicit}
Let $f\in L^2_+(\R )$ and $t\in \R$, then the function $\Omega_tf$ defined by
\[
\forall z\in \C_+,\quad
 \Omega_tf(z):=\frac{1}{2\ii\pi}I_+((X^*-2tL_{u_0}-z\,\mathrm{Id})^{-1}f)\]
belongs to $L^2_+(\R)$, and $\| \Omega_tf\|_{L^2}\leq \| f\|_{L^2}$. In particular, for every $z\in\C_+$, we have
\[ |\Omega_tf(z)|\leq \frac{\| f\|_{L^2}}{2\sqrt{\pi \mathrm{Im}z}}. \]
\end{lemma}
\begin{proof}
If $u_0\in H^2(\R)$, the proof of the explicit formula in~\cite{Gerard22} gives
\[\Omega_t =U(t)\mathrm{e}^{\ii tL_{u_0}^2},\]
where $U(t)$ is the --- unitary--- solution of the following linear ODE in the space of bounded operators on $L^2_+(\R)$, 
\[ U'(t)=B_{u(t)}U(t),\quad U(0)=\mathrm{Id}, \]
and $B_u=\ii(T_{|D|u}-T_u^2)$ is the second part of the Benjamin--Ono Lax pair. The general case follows by mollifying $u_0$. As for the last inequality, it follows from the inverse Fourier transform applied to $g=\Omega_tf$,
\[ \forall z\in \C_+,
\quad g(z)=\int_0^\infty \mathrm{e}^{\ii z\xi}\widehat g(\xi )\, \frac{\dd \xi}{2\pi},\]
and the Cauchy--Schwarz inequality combined with the Plancherel identity complete the proof.
\end{proof}
\begin{lemma}\label{lem:cancel}
Let $j\in \{ 1,\dots, N\}$ and $f\in L^2_+(\R)$ be such that $f\perp \phe_j$. Then, with the notation of Lemma \ref{L2boundexplicit}, for every $z\in\C_+$,
\[ \Omega_tf(z-2\lambda_jt)\td_t,\infty 0.\]
\end{lemma}
\begin{proof}
By functional calculus for the selfadjoint operator $L_{u_0}$, there exists $h\in L^2_+(\R )$ such that 
\[(L_{u_0}-\lambda_j\mathrm{Id})h=f.\]
First assume that $f\in D(X^*)$, and let us prove that $h\in D(X^*)$. Using the --- standard ---  Fourier transform, we have
\[ \forall \xi \in ]0,+\infty[,
\quad
 (\xi -\lambda_j)\widehat h(\xi )=\widehat f(\xi )+\frac{1}{2\pi}\widehat u_0* \widehat h(\xi ).\]
Since $f\in D(X^*)$, $\widehat f\in H^1(0,+\infty)$. Moreover, since $xu_0\in L^2(\R)$, $\widehat u_0\in H^1(\R)$, hence $\widehat u_0*\widehat h$ is bounded on $\R $ and its derivative is bounded on $\R$. Consequently,
\[ \widehat h(\xi )=\frac{\widehat f(\xi)}{\xi -\lambda_j}+\frac{\widehat u_0* \widehat h(\xi )}{2\pi (\xi -\lambda_j)}\]
belongs to $H^1(0,\infty )$, namely $h\in D(X^*)$. Let us complete the proof. We have
\begin{align*}
(X^*-2t(L_{u_0}-\lambda_j\mathrm{Id})-z\,\mathrm{Id})^{-1}f&=(X^*-2t(L_{u_0}-\lambda_j\mathrm{Id})-z\,\mathrm{Id})^{-1}(L_{u_0}-\lambda_j\mathrm{Id})h\\
&=-\frac{h}{2t}+\frac{1}{2t}(X^*-2t(L_{u_0}-\lambda_j\mathrm{Id})-z\,\mathrm{Id})^{-1}(X^*h-zh).
\end{align*}
Consequently, using Lemma~\ref{L2boundexplicit},
\[ \Omega_tf(z-2t\lambda_j)=-\frac{I_+(h)}{4\ii\pi t}+\frac{\Omega_t(X^*h-zh)(z-2t\lambda_j)}{2t}\td_t,\infty 0.\]
Finally, for a general $f\in \phe_j^\perp$, the result follows from the uniform bound on $\Omega_tf(z)$ provided by Lemma \ref{L2boundexplicit}, and from the elementary fact that, since $D(X^*)$ is dense in $L^2_+(\R)$, $D(X^*)\cap \phe_j^\perp$ is dense in $\phe_j^\perp$.
\end{proof}

We also need to know a little more about the eigenfunction $\phe_j$.
\begin{lemma}\label{X*phi}
We have $\phe_j\in D(X^*)$. 
\end{lemma}
\begin{proof}
Using again the Fourier transform, we have, for every $\xi >0$,
\begin{align*}
 \widehat \phe_j(\xi )&=\frac{\widehat u_0*\widehat \phe_j(\xi )}{2\pi (\xi -\lambda_j)},\\
 \widehat \phe_j'(\xi )&=-\frac{\widehat u_0*\widehat \phe_j(\xi) }{2\pi (\xi -\lambda_j)^2}+\frac{(\widehat u_0)'*\widehat \phe_j(\xi)}{2\pi (\xi -\lambda_j)},\\
  \end{align*}
and the right hand side in each line clearly belongs to $L^2(0,+\infty)$. 
\end{proof}

Let us now prove the soliton limit~\eqref{soliton}.
\begin{proof}[Proof of the soliton limit~\eqref{soliton}] First we write
\[\Pi u_0=\langle \Pi u_0, \phe_j\rangle \phe_j +f_j,\]
where $f_j\in \phe_j^\perp $. From Lemma \ref{lem:cancel}, we infer, for every $z\in \C_+$,  $\Omega_tf_j(z-2t\lambda_j)\to 0$ as $t\to \infty$. Next we set
\[ \psi_j :=\frac{\phe_j}{\langle X^*\phe_j, \phe_j\rangle -z}\]
and we observe that
\[ (X^*-2t(L_{u_0}-\lambda_j\mathrm{Id})-z\,\mathrm{Id})\psi_j=\phe_j +g_j\ ,\quad g_j:=\frac{X^*\phe_j -\langle X^*\phe_j,\phe_j\rangle \phe_j}{\langle X^*\phe_j,\phe_j\rangle -z}.\]
Consequently,
\[ (X^*-2t(L_{u_0}-\lambda_j\mathrm{Id})-z\,\mathrm{Id})^{-1}\phe_j=\psi_j-(X^*-2t(L_{u_0}-\lambda_j\mathrm{Id})-z\,\mathrm{Id})^{-1}g_j,\]
and
\[ \Omega_t(\phe_j )(z-2t\lambda_j)=\frac{1}{2\ii \pi}I_+(\psi_j)-\Omega_t(g_j)(z-2t\lambda_j).\]
Since $g_j\in \phe_j^\perp$, we  again infer from Lemma \ref{lem:cancel} that  
$ \Omega_t(\phe_j)(z-2t\lambda_j)\to \frac{1}{2\ii \pi}I_+(\psi_j)$ as $t\to\infty$.
Finally, we get
\[ \Omega_t(\Pi u_0)(z-2t\lambda_j)\to \frac{\langle \Pi u_0,\phe_j\rangle }{2\ii \pi}I_+(\psi_j)=\frac{\langle \Pi u_0,\phe_j\rangle I_+(\phe_j)}{2\ii \pi(\langle X^*\phe_j, \phe_j\rangle -z)}=\frac{\ii}{z-\langle X^*\phe_j,\phe_j\rangle },\]
where the last equality is obtained thanks to Wu's identities~\eqref{eq:Wu}. This is exactly the soliton limit \eqref{soliton}.
\end{proof}

\section{Setting the radiation limit}\label{sec:radiationlimit}

Since the radiation limit property~\eqref{scat} must be shown in the sense of weak convergence in the variable $\lambda$, we can assume that $\lambda$ is restricted to a compact set $\Lambda=[\Lambda_-,\Lambda_+]\subset(0,+\infty)$. 
 Let us introduce a cutoff function $\chi$ equal to $1$ on $[\Lambda_-/2,2\Lambda_+]$ and supported in $[\Lambda_-/4,4\Lambda_+]$. 

In this section, we establish the following proposition.

\begin{proposition}[Sufficient condition for~\eqref{scat}]\label{redscat}
Recall the generalized eigenfunction $m_-(x,\lambda)$ defined in \eqref{def:m}, recall $f_{t,z}$ from \eqref{explicit} and set $q_{t,z}:=\Pi u_0-2tT_{u_0}f_{t,z}$.  Property~\eqref{scat} holds if the following three convergence results are satisfied.
\begin{itemize}
\item Localization. For every $z\in\C_+$, we have the following weak convergence in $L^2(\Lambda)$:
\begin{equation}\label{loc}
\sqrt{2t}\int_0^{\infty}\ee^{\ii t (\xi-\lambda)^2+\ii z \xi}\widehat{q}_{t,z-2t\lambda}(\xi)(1-\chi(\xi))\frac{\dd\xi}{2\pi}\tdw_t,{+\infty} 0.
\end{equation}
\item Main term. For every $z\in\C_+$, we have the following strong convergence in $L^2(\Lambda)$:
\begin{equation}\label{distorted}
\sqrt{2t}\int_0^{\infty}\ee^{\ii t (\xi-\lambda)^2+\ii z\xi}\int_{\R}\chi (\xi)u_0(x)\overline{m_-}(x,\xi)\dd x\frac{\dd\xi}{2\pi}
	\td_t,{+\infty} \frac{\mathrm{e}^{\ii\pi/4}}{\sqrt{2\pi}}\ee^{\ii\lambda z}\widetilde{\Pi u_0}(\lambda).
\end{equation}
\item Remainder term.   For every $z\in\C_+$, we have the following weak convergence in $L^2_w(\Lambda)$:
\begin{equation}\label{cancel}
\int_\R q_{t,z-2t\lambda}(x)\int_\R \mathrm{e}^{\ii\frac{\alpha ^2}{2}+\ii\frac{z\alpha }{\sqrt{2t}}}\chi \left (\lambda +\frac{\alpha}{\sqrt{2t}}\right )\overline r_{t,z,\lambda}(x,\alpha)\, \frac{\dd \alpha}{2\pi}\dd x 	
	\tdw_t,{+\infty} 0,
\end{equation}
with
\begin{equation}\label{r}
r_{t,z,\lambda}(x,\alpha)
	:=\ii\int_{-\infty}^x \mathrm{e}^{\ii\lambda (x-y)}T_{u_0}m_-\left (y , \lambda +\frac{\alpha}{\sqrt{2t}}\right )\left [ \mathrm{e}^{-\ii\overline z\frac{(x-y)}{2t}+\ii\frac{(x^2-y^2)}{4t}}- \mathrm{e}^{\ii\alpha \frac{(x-y)}{\sqrt{2t}}}  \right ]\dd y.
\end{equation}
\end{itemize}
\end{proposition}
Property \eqref{distorted} seems accessible because of the decay assumed on $u_0$ and since we have good estimates on $m_-(x,\lambda)$ as $\lambda $ stays in a segment of $(0,+\infty)$. It is established in Section \ref{3.8}, as well as the localization near $\Lambda$. Property \eqref{cancel} is established in Section \ref{3.9}. It is more delicate to handle, since one must use appropriate estimates on $q_{t,z-2t\lambda}$ as $t\to +\infty$. These estimates are derived in Section \ref{sec:estimates}. First we prove Proposition \ref{redscat}.

\subsection{Choosing the unknown function}

Given $z\in \C_+$ and $\lambda >0$, we study (see \eqref{explicit})
\[ \Pi u(t, z-2t\lambda )=\frac{1}{2\ii\pi}I_+(f_{t,z-2t\lambda})=\frac{1}{2\ii\pi}\widehat f_{t,z-2t\lambda}(0^+).\]
We set, for $\xi >0$, 
\[ h_{t,z,\lambda}(\xi ):=(2t)^{1/2}\mathrm{e}^{\ii t(\xi -\lambda)^2}\widehat f_{t,z-2t\lambda}(\xi ).\]
From the equation $(X^*-2t(D-\lambda\,\mathrm{Id})+2tT_{u_0}-z\,\mathrm{Id})f_{t,z-2t\lambda}=\Pi u_0$ expressed in the Fourier transform and in the domain $\xi >0$, we infer
\begin{equation}\label{hp}
 h_{t,z,\lambda}(\xi )=\ii(2t)^{1/2}\int_\xi ^\infty \mathrm{e}^{\ii t(\eta -\lambda)^2+\ii z(\eta -\xi)}\widehat q_{t,z,\lambda}(\eta )\dd \eta,
 \end{equation}
where  $q_{t,z-2t\lambda }\in L^2_+$ can be written in the form
 \begin{equation}\label{qf}
q_{t,z-2t\lambda}=(X^*-2t(D-\lambda\,\mathrm{Id} )-z\,\mathrm{Id})f_{t,z-2t\lambda}=\Pi u_0-2tT_{u_0}f_{t,z-2t\lambda}.
 \end{equation}
 Let us look for an equation for $q_{t,z-2t\lambda}$. We have, from \eqref{hp},
 \[ 2t \widehat f_{t,z-2t\lambda}(\xi )=(2\ii t)\mathrm{e}^{-\ii t(\xi -\lambda)^2}\int_\xi ^\infty \mathrm{e}^{\ii t(\zeta -\lambda )^2+\ii z(\zeta -\xi)}\widehat q_{t,z-2t\lambda}(\zeta )\dd \zeta ,\]
and consequently
 \begin{align*}
 2tf_{t,z-2t\lambda}(x)&= 2\ii t \int_\R q_{t,z-2t\lambda}(y) \dd y \int_0^\infty \frac{\dd\xi}{2\pi}\int_\xi ^\infty \dd\zeta \, \mathrm{e}^{\ii t(\zeta -\xi )(\zeta +\xi -2\lambda)+\ii z(\zeta -\xi)+\ii x\xi -\ii\zeta y}     \\
 &= \ii \int_\R q_{t,z-2t\lambda}(y) \dd y \int_0^\infty \frac{\dd\xi}{2\pi}\int_0 ^\infty \dd\sigma \, \mathrm{e}^{\ii\sigma (\xi -\lambda+(z-y)/(2t)+\sigma/(4t) )+\ii(x-y)\xi }\\
 &=\ii\int_0^\infty \mathrm{e}^{-\ii\lambda \sigma +\ii\sigma (2z+\sigma)/(4t)}\left [S\left (\frac{\sigma}{2t}\right )^*q_{t,z-2t\lambda}\right ](x+\sigma )\dd \sigma ,
 \end{align*}
 where $(S(\eta )^*)_{\eta \geq 0}$ denotes the adjoint Lax--Beurling semigroup on $L^2_+(\R )$,
 \[ S(\eta)^*f=\mathrm{e}^{-\ii\eta X^*}f=\Pi [\mathrm{e}^{-\ii x\eta}f(x)].\]
The equation for $q_{t,z-2t\lambda}$ therefore reads
\begin{equation}\label{inteqqtbis}
q_{t,z-2t\lambda}(x)=\Pi u_0(x)-\ii\int_0^\infty \mathrm{e}^{-\ii\lambda \sigma +\ii\sigma (2z+\sigma)/(4t)} T_{u_0}\left [S\left (\frac{\sigma}{2t}\right )^*q_{t,z-2t\lambda}(\diamond +\sigma )\right ](x)\dd \sigma.
\end{equation}

\subsection{Making the distorted Fourier transform appear} Coming back to the explicit formula and to \eqref{hp}, we must study the limit as $t\to \infty $ of
\begin{equation}\label{Piuq}
(2t)^{1/2} \mathrm{e}^{\ii t\lambda ^2}\Pi u(t,z-2t\lambda )=\frac{1}{2\ii \pi}h_{t,z,\lambda}(0^+)
=(2t)^{1/2}\int_0^\infty \mathrm{e}^{\ii t(\eta-\lambda)^2+\ii z\eta}\widehat q_{t,z-2t\lambda}(\eta)\, \frac{\dd\eta}{2\pi}.
\end{equation}
We first localize near the critical point $\lambda$  the domain of the oscillatory integral in the right--hand side, by using~\eqref{loc}.

Now, using the following equation defining the generalized eigenfunction $m_-$ of $L_{u_0}$, 
\begin{equation}\label{m}
\forall \eta >0,\quad
 m_-(x,\eta)=\mathrm{e}^{\ii \eta x}+\ii \int_{-\infty}^x \mathrm{e}^{\ii\eta (x-y)}T_{u_0}m_-(y,\eta)\dd y,
\end{equation} 
we get
\begin{equation*}
\widehat q_{t,z-2t\lambda}\left (\eta\right )
	=\int_\R q_{t,z-2t\lambda}(x)\overline{m_-}\left (x,\eta\right )\dd x
	+\ii \int_\R q_{t,z-2t\lambda}(x)\int_{-\infty}^x \mathrm{e}^{\ii \eta(x-y)}\overline{T_{u_0}m_-}\left (y,\eta\right )\dd y \dd x.
\end{equation*}
Using the equation \eqref{inteqqtbis} satisfied by $q_{t,z-2t\lambda}$, we finally obtain
\begin{multline}\label{qhat}
 \widehat q_{t,z-2t\lambda}\left (\eta\right )=\int_\R \Pi u_0(x)\overline{m_-}\left (x,\eta\right )\dd x	\\
	-\ii \int_\R q_{t,z-2t\lambda}(x) \int_0^\infty\overline{T_{u_0}m_-}\left (x-\sigma , \eta\right )\left [ \mathrm{e}^{-\ii\lambda\sigma-\ii \sigma (2 z+\sigma -2x)/(4t)}- \mathrm{e}^{-\ii \eta \sigma }  \right ]\dd \sigma
\dd x.
\end{multline}
Notice that, in the first integral of the right--hand side, $\Pi$ can be dropped since $m_-\in L^\infty_+$.  
In the second integral of the right--hand side, we make the change of variable $x-\sigma=y$ and $\xi =\lambda +\frac{\alpha}{\sqrt{2t}}.$ 
Hence the decomposition  \eqref{qhat} transforms~\eqref{Piuq} into
\begin{align*}
\sqrt{2t}\ee^{\ii t \lambda^2}\Pi u(t,z-2t\lambda)
	=&\sqrt{2t}\int_0^{\infty}\ee^{\ii t (\xi-\lambda)^2+\ii z\xi} \widehat{q}_{t,z-2t\lambda}(\xi)(1-\chi (\xi))\frac{\dd\xi}{2\pi}\\ \nonumber
	&+\sqrt{2t}\int_0^{\infty}\ee^{\ii t (\xi-\lambda)^2+\ii z\xi}\int_{\R}\chi (\xi)u_0(x)\overline{m_-}(x,\xi)\dd x\frac{\dd\xi}{2\pi}
	\\ \nonumber
	&+\ee^{\ii \lambda z}\int_{-\infty}^{\infty}\ee^{\ii \alpha^2/2+\ii \frac{z\alpha}{\sqrt{2t}} }\int_{\R}\chi \left(\lambda+\frac{\alpha}{\sqrt{2t}}\right)q_{t,z-2t\lambda}(x)\overline{r}_{t,z,\lambda}(x,\alpha)\dd x\frac{\dd \alpha}{2\pi},
\end{align*}
where $r_{t,z,\lambda}(x,\alpha)$ satisfies~\eqref{r}.
We are therefore reduced to proving~\eqref{loc}, \eqref{distorted} and~\eqref{cancel}. This completes the proof of Proposition \ref{redscat}.

\section{A priori estimates}\label{sec:estimates}

In this section, we derive a priori estimates on 
\begin{equation}\label{def:f}
f_{t,z}=(X^*-2tL_{u_0}-z\,\mathrm{Id})^{-1}\Pi u_0
\end{equation}
 and $q_{t,z}=\Pi u_0-2tT_{u_0}f_{t,z}$, and their translates in $z$ by $2t\lambda$, $\lambda>0$.
 The Fourier transform of $f_{t,z}$ satisfies the equation
 \begin{equation}\label{eq:fhat-no-translation}
 \forall \xi>0,\quad\left(\ii\frac{\dd}{\dd\xi}-2t\xi-z\,\mathrm{Id}\right)\widehat{f}_{t,z}(\xi)+\frac{t}{\pi}\widehat{u_0}\ast\widehat{f}_{t,z}(\xi)=\widehat{u_0}(\xi).
 \end{equation}

\subsection{$L^2$ bounds}

\begin{lemma}
For each $t\in\mathbb{R}$ and $z\in\mathbb{C}_+$, the function $\widehat{f}_{t,z}$ defined by~\eqref{eq:fhat} satisfies $\xi\widehat{f}_{t,z}\in L^2(0,\infty)$ and $\widehat{f}_{t,z}\in H^1(0,\infty)$.
\end{lemma}

\begin{proof}
We first show that 
\[
g^\eps_{t,z}
	:=[L_{u_0}(\mathrm{Id}+\eps L_{u_0})^{-1},(X^*-2tL_{u_0}-z\,\mathrm{Id})]f_{t,z}
\]
stays bounded in $L^2_+(\R)$ as $\eps\to 0$ with $\eps$ positive.
Note that $g^\epsilon_{t,z}$ can be equivalently written in the form
\[
g^\eps_{t,z}
	=[L_{u_0}(\mathrm{Id}+\eps L_{u_0})^{-1},X^*]f_{t,z}
\]
 and write
\[
L_{u_0}(\mathrm{Id}+\eps L_{u_0})^{-1}=\frac{1}{\eps}\left(\mathrm{Id}-(\mathrm{Id}+\eps L_{u_0})^{-1}\right),
\]
so that
\begin{align*}
[L_{u_0}(\mathrm{Id}+\eps L_{u_0})^{-1},X^*]
	&=-\frac{1}{\eps}[(\mathrm{Id}+\eps L_{u_0})^{-1},X^*]
	\\
	&=-(\mathrm{Id}+\eps L_{u_0})^{-1}[X^*,L_{u_0}](\mathrm{Id}+\eps L_{u_0})^{-1}.
\end{align*}
Using the identity $[X^*,L_{u_0}]=\ii \,\mathrm{Id}-\frac{\ii}{2\pi}\Pi u_0 I_+(\diamond)$ (see~\cite[Lemma 10]{Gerard22}),
we then decompose
\begin{equation*}
g^\eps_{t,z}
	=-\ii (\mathrm{Id}+\eps L_{u_0})^{-2}f_{t,z}+\frac{\ii}{2\pi}(\mathrm{Id}+\eps L_{u_0})^{-1}(\Pi u_0) I_+((\mathrm{Id}+\eps L_{u_0})^{-1}f_{t,z}).
\end{equation*}
Since $L_{u_0}$ is bounded from below, the operator $\mathrm{Id}+\eps L_{u_0}$ is invertible for small, positive values of $\eps$, with a bounded inverse as $\eps\to 0$. Therefore the first term in the right-hand side is bounded in $L^2_+(\R)$ as $\eps\downarrow 0$. As for the second term in the right-hand side, $(\mathrm{Id}+\eps L_{u_0})^{-1}(\Pi u_0)$ is bounded in $L_+^2(\mathbb{R})$ as $\epsilon\downarrow 0$ for the same reason, so it is enough to show that the scalar $I_+((\mathrm{Id}+\eps L_{u_0})^{-1}f_{t,z})$ remains bounded as $\eps\downarrow 0$.  For that, we begin with the operator identity
\[
(\mathrm{Id}+\eps L_{u_0})^{-1}=(\mathrm{Id}+\eps D)^{-1}+\eps (\mathrm{Id}+\eps D)^{-1}T_{u_0}(\mathrm{Id}+\eps L_{u_0})^{-1}.
\]
Since the symbol $\xi$ of $D$ vanishes as $\xi\downarrow 0$, by definition of $I_+$, we deduce
\[
 I_+((\mathrm{Id}+\eps L_{u_0})^{-1} f_{t,z})
 	= I_+(f_{t,z}) 
 	+\eps I_+(T_{u_0}((\mathrm{Id}+\eps L_{u_0})^{-1}f_{t,z})).
\]
Note that $\widehat{f}_{t,z}\in H^1(0,\delta)$ for every $\delta>0$ (see below Theorem 6 in~\cite{Gerard22}), therefore $ I_+(f_{t,z})$ is well-defined. Moreover, for any $k\in L_+^2(\mathbb{R})$ and real-valued $u_0\in L^2(\mathbb{R})$ we have $I_+(T_{u_0}k)=\langle k,\Pi u_0\rangle$, so taking $k=(\mathrm{Id}+\eps L_{u_0})^{-1}f_{t,z}$ we have
\[
|I_+(T_{u_0}((\mathrm{Id}+\eps L_{u_0})^{-1}f_{t,z}))|
	=|\langle (\mathrm{Id}+\eps L_{u_0})^{-1}f_{t,z},\Pi u_0\rangle|
	\leq C\| f_{t,z}\|_{L^2}\|\Pi u_0\|_{L^2}.
\] 
Combining the two observations together, we get that the term $ I_+((\mathrm{Id}+\eps L_{u_0})^{-1}f_{t,z})
$ is bounded as $\eps\downarrow 0$.
We conclude   that $g^\eps_{t,z}$ is bounded in $L^2_+$ as $\eps\downarrow 0$.

Finally, we observe using~\eqref{def:f} that
\[
g^\eps_{t,z}
	=L_{u_0}(\mathrm{Id}+\eps L_{u_0})^{-1}\Pi u_0
	-(X^*-2tL_{u_0}-z\,\mathrm{Id})L_{u_0}(\mathrm{Id}+\eps L_{u_0})^{-1}f_{t,z}.
\]
Since $u_0\in H^1(\R)$, $\Pi u_0$ belongs to the domain $H^1_+(\R)$ of $L_{u_0}$. Therefore $L_{u_0}(\mathrm{Id}+\eps L_{u_0})^{-1}\Pi u_0$ is bounded in $L^2_+(\R)$ as $\eps\downarrow 0$. Now, $(X^*-2tL_{u_0}-z\,\mathrm{Id})^{-1}$ is a bounded operator on $L_+^2(\mathbb{R})$ --- see  Remark \ref{invertibility} ---, so applying it to the above identity, we conclude that  $L_{u_0}(\mathrm{Id}+\eps L_{u_0})^{-1}f_{t,z}$ is bounded in $L^2_+(\R)$ as $\eps\downarrow 0$. If, for any $(t,z)\in\mathbb{R}\times\mathbb{C}_+$, we denote by $\mu_{t,z} $ the spectral measure of the self--adjoint operator $L_{u_0}$ associated with  the vector $f_{t,z}\in L^2_+(\mathbb{R})$, we have
\[ \Vert L_{u_0}(\mathrm{Id}+\eps L_{u_0})^{-1}f_{t,z}\Vert_{L^2}^2=\int_\R \lambda^2(1+\eps \lambda)^{-2}\, \dd\mu_{t,z} (\lambda),\]
and therefore, by Fatou's lemma,
\[ \int_\R \lambda^2\, \dd\mu_{t,z} (\lambda)<+\infty ,\]
which precisely means that $f_{t,z}$ belongs to the domain of the operator $L_{u_0}$, namely $H^1_+(\R)$.
 We then conclude that $\xi\widehat{f}_{t,z} \in L^2(0,\infty)$ and using equation~\eqref{eq:fhat-no-translation} with Young's inequality for $\widehat{u_0}\in L^1$ and $\widehat{f}_{t,z}\in L^2$, we also get that $\widehat{f}_{t,z}\in H^1(0,\infty)$.
\end{proof}


\begin{lemma}\label{lem:f-bound}
Let $\Lambda=[\Lambda_-,\Lambda_+]\subset (0,\infty)$.  There is $C>0$ such that when $\operatorname{Im}z+|\operatorname{Re}(z)|/t$ is small enough and $t\ge 1$, then 
$(x,\lambda)\mapsto f_{t,z-2t\lambda}(x)$ satisfies
\[
t  \|f_{t,z-2t\lambda}\|_{L^2_{x,\lambda}(\R\times\Lambda)}^2
	\leq C.
\]
\end{lemma}

\begin{proof}
Replacing $z$ with $z-2t\lambda$ in \eqref{eq:fhat-no-translation} gives the equation
\begin{equation}\label{eq:fhat}
\forall \xi >0,\quad \left(\ii \frac{\dd }{\dd \xi}-2t\xi+(2t\lambda-z)\mathrm{Id}\right)\widehat{f}_{t,z-2t\lambda}(\xi) +\frac{t}{\pi} \widehat{u_0}\ast \widehat{f}_{t,z-2t\lambda}(\xi)=\widehat u_0(\xi).
\end{equation}
We integrate identity~\eqref{eq:fhat} against $\overline{(\widehat{f}_{t,z-2t\lambda})'}$ and take the real part. We find that
\[
\operatorname{Re}\langle \ii (\widehat{f}_{t,z-2t\lambda})',{( \widehat{f}_{t,z-2t\lambda})'}\rangle=0.
\]
Then, by integration by parts,
\[
\operatorname{Re}\langle -2t\xi  \widehat{f}_{t,z-2t\lambda}, {(\widehat{f}_{t,z-2t\lambda})'}\rangle
	=-t \left[\xi|\widehat{f}_{t,z-2t\lambda}|^2\right]_0^{+\infty}+t\Vert \widehat{f}_{t,z-2t\lambda}\Vert_{L^2}^2.
\]
Since $\widehat{f}_{t,z-2t\lambda}\in H^1(0,+\infty)$, it is continuous at $\xi=0^+$. Moreover, since $\xi \widehat{f}_{t,z-2t\lambda}$ and $(\widehat{f}_{t,z-2t\lambda})'$ belong to $L^2(0,+\infty)$, using that 
\[
\xi |\widehat{f}_{t,z-2t\lambda}|^2=-\operatorname{Re}\int_\xi^{\infty} \left(2\eta \widehat{f}_{t,z-2t\lambda}(\eta)\overline{(\widehat{f}_{t,z-2t\lambda})'(\eta)}+ |\widehat{f}_{t,z-2t\lambda}(\eta)|^2 \right)\dd\eta,
\]
 we deduce that $\xi |\widehat{f}_{t,z-2t\lambda}|^2\to 0$ as $\xi\to+\infty$, so
\[
\operatorname{Re}\langle -2t\xi \widehat{f}_{t,z-2t\lambda}, {( \widehat{f}_{t,z-2t\lambda})'}\rangle
	=t\| \widehat{f}_{t,z-2t\lambda}\|_{L^2}^2.
\]

Next, we have by direct integration that
\[
\operatorname{Re}\langle(2t\lambda-\operatorname{Re}(z)) \widehat{f}_{t,z-2t\lambda}, {(\widehat{f}_{t,z-2t\lambda})'}\rangle
	=-(t\lambda-\operatorname{Re}(z)/2)|I_+(f_{t,z-2t\lambda})|^2.
\]
We note that thanks to the explicit formula (see \eqref{explicit}),
\[
|I_+(f_{t,z-2t\lambda})|
	=2\pi |\Pi u(t,z-2t\lambda)|,
\]
hence the conservation of the $L^2$ norm for the Benjamin-Ono equation implies that
\begin{equation}\label{eq:I+bound}
\int_{\R_+} |I_+(f_{t,z-2t\lambda})|^2\dd\lambda 
	=(2\pi)^2\int_{\R_+}|\Pi u(t,z-2t\lambda)|^2\dd\lambda
	\leq \frac{C}{t}.
\end{equation}

We now focus on
\[
2t\operatorname{Re}\langle \widehat{u_0}\ast \widehat{f}_{t,z-2t\lambda}, {(\widehat{f}_{t,z-2t\lambda})'}\rangle
	=\left[t(\widehat{u_0}\ast \widehat{f}_{t,z-2t\lambda})\overline{ \widehat{f}_{t,z-2t\lambda}}\right]_0^{+\infty},
\]
where we used integration by parts and the fact that $u_0$ is real-valued. Hence we get the upper bound 
\begin{align*}
\left|2t\operatorname{Re}\langle \widehat{u_0}\ast \widehat{f}_{t,z-2t\lambda}, {(\widehat{f}_{t,z-2t\lambda})'}\rangle\right|
	&\leq Ct|\langle u_0,f_{t,z-2t\lambda}\rangle| |I_+(f_{t,z-2t\lambda})|\\
	&\leq t(\epsilon|\langle u_0,f_{t,z-2t\lambda}\rangle|^2+C_\epsilon |I_+(f_{t,z-2t\lambda})|^2).
\end{align*}

To conclude, integrating~\eqref{eq:fhat} against $\overline{(\widehat{f}_{t,z-2t\lambda})'}$ and taking the real part, we have shown that for large $t$, if $\operatorname{Re}(z)/t$ stays bounded,
\begin{equation}\label{eq:f-int}
\int_{\Lambda}\left|t\|\widehat{f}_{t,z-2t\lambda}\|_{L^2}^2
 +\operatorname{Im}(z)\operatorname{Im}\langle \widehat{f}_{t,z-2t\lambda},{(\widehat{f}_{t,z-2t\lambda})'}\rangle\right|\dd\lambda
 	\leq C_{\epsilon,\Lambda} + t\epsilon \|u_0\|_{L^2}^2\int_{\Lambda}\|f_{t,z-2t\lambda}\|_{L^2}^2\dd\lambda.
\end{equation}

If $\operatorname{Im}(z)= 0$, we can absorb the second term on the right-hand side into the left-hand side.
Otherwise, we need to deal with the term $\operatorname{Im}(z)\operatorname{Im}\langle\widehat{f}_{t,z-2t\lambda},{(\widehat{f}_{t,z-2t\lambda})'}\rangle$ separately.
For this we integrate~\eqref{eq:fhat} against $\overline{ \widehat{f}_{t,z-2t\lambda}}$ and take the imaginary part: we find
\[
\operatorname{Im}\langle \widehat{f}_{t,z-2t\lambda}, {( \widehat{f}_{t,z-2t\lambda})'}\rangle
	-\operatorname{Im}(z)\|\widehat{f}_{t,z-2t\lambda}\|_{L^2}^2
	-\frac{t}{\pi}\operatorname{Im} \langle \widehat{u_0}\ast  \widehat{f}_{t,z-2t\lambda}, \widehat{f}_{t,z-2t\lambda}\rangle 
	=\operatorname{Im} \langle \widehat{u_0},\widehat{f}_{t,z-2t\lambda}\rangle.
\]
We note that Young's inequality implies
\[
t| \langle \widehat{u_0}\ast \widehat{f}_{t,z-2t\lambda},\widehat{f}_{t,z-2t\lambda}\rangle|
	\leq Ct \|\widehat{u_0}\|_{L^1}\| \widehat{f}_{t,z-2t\lambda}\|_{L^2}^2,
\]
where  $\widehat{u_0}\in L^1$ if $u_0$ since $u_0\in H^{1}$. Hence when $|\operatorname{Im}(z)|\leq \epsilon$ we find
\[
|\operatorname{Im}(z)\operatorname{Im}\langle \widehat{f}_{t,z-2t\lambda},(\widehat{f}_{t,z-2t\lambda})'\rangle|
	\leq C\epsilon (t \| \widehat{f}_{t,z-2t\lambda}\|_{L^2}^2+1),
\]
so we can still absorb this term in~\eqref{eq:f-int} to find an upper bound on $t\| \widehat{f}_{t,z-2t\lambda}\|_{L^2_{\lambda,\xi}(\Lambda\times(0,\infty))}^2$.
\end{proof}

From the $L^2$ bound on $(x,\lambda)\mapsto f_{t,z-2t\lambda}(x)$ in Lemma~\ref{lem:f-bound} we deduce a $L^2$ bound on $(x,\lambda)\mapsto X^* q_{t,z-2t\lambda}(x)$.

\begin{lemma}\label{lem:q-bound}
For all subintervals $\Lambda=[\Lambda_-,\Lambda_+]\subset(0,\infty)$ we have
\begin{equation}\label{qH1}
\| X^*q_{t,z-2t\lambda}\|_{L^2_{x,\lambda}(\R\times\Lambda)} 
	 \leq C\sqrt t.
\end{equation}
\end{lemma}

\begin{proof}
Combining Lemma \ref{lem:f-bound} with \eqref{qf} leads to
\begin{equation}\label{qL2}
\Vert q_{t,z-2t\lambda}\Vert _{L^2_{x,\lambda}(\R\times\Lambda)}\leq C\sqrt{t}.
\end{equation}
Furthermore, since $X^*T_{u_0}f=T_{xu_0}f$ and $xu_0\in H^1\subset L^\infty $,  Lemma \ref{lem:f-bound} and \eqref{qf} lead to the upper bound on $\| X^*q_{t,z-2t\lambda}\|_{L^2(\R\times\Lambda)}$.
\end{proof}
\begin{remark}\label{xuH1} The above proof of Lemma \ref{lem:q-bound} is the only place in our arguments where we use specifically the assumption $xu_0\in H^1$, and only in the weaker form $xu_0\in L^\infty$. In particular the case of the initial datum $u_0(x)=c{\bf 1}_{(-1,1)}(x)$ can be treated.
\end{remark}
\subsection{Weak convergence}

\begin{lemma}\label{lem:q-I+}
Let $\Lambda:=[\Lambda_-,\Lambda_+]\subset(0,\infty)$.  The following properties hold. 
\begin{enumerate} 
\item For every bounded family $(\psi_t)_t\in L^\infty (\Lambda)$, for every function $\phe \in L^2_+(\R)$, 
\[\int_{\R\times \Lambda}\sqrt t f_{t,z-2t\lambda}(x)\phe (x)\psi_t(\lambda)\dd x\dd \lambda \td_t,{+\infty} 0.\]
\item For every $\psi \in L^2(\Lambda)$, as $t\to+\infty$,
\[\int_\Lambda I_+(q_{t,z-2t\lambda})\psi (\lambda)\dd \lambda =o(\sqrt{t}).\]
\end{enumerate}
\end{lemma}

\begin{proof} Let us prove the first property. We already know from Lemma \ref{lem:f-bound} that the family $((\xi,\lambda)\mapsto\sqrt{t}\widehat{f}_{t,z-2t\lambda})_{t}$ is bounded in $L^2_{\xi,\lambda}( \R\times\Lambda)$, so that $\widehat{F}_{t,z,\lambda}:=\psi_t\sqrt{t}\widehat{f}_{t,z-2t\lambda}$ is bounded as well. Hence up to a subsequence, we can assume the weak convergence $\widehat{F}_{t,z,\lambda}\rightharpoonup \widehat{F}_{\infty, z,\lambda}$  in $L^2_{\xi,\lambda}( \R\times\Lambda)$ as $t\to+\infty$. Then passing to the weak limit in~\eqref{eq:fhat} after dividing by $\sqrt{t}$, we find that $\widehat{F}_{\infty, z,\lambda}$ is a solution to
\[
-2(\xi-\lambda)\widehat{F}_{\infty,z,\lambda}+\frac{1}{\pi}\widehat{u_0}\ast \widehat{F}_{\infty,z,\lambda }=0.
\]
We recognize this as the Fourier transform of the equation $L_{u_0}F_{\infty,z,\lambda}=\lambda F_{\infty,z,\lambda}$. Taking the distorted Fourier transform of both sides, we find that in $L^2_{\eta,\lambda}( (0,\infty)\times\Lambda)$,
\[
\eta \widetilde{F}_{\infty,z,\lambda}(\eta)=\lambda \widetilde{F}_{\infty,z,\lambda}(\eta).
\]
This implies that $\widetilde{F}_{\infty,z,\lambda}= 0$ outside the diagonal $\lambda=\eta$. Since $\widetilde{F}_{\infty,z,\lambda}\in L^2_{\eta,\lambda}( (0,\infty)\times\Lambda)$, we conclude that $\widetilde{F}_{\infty,z,\lambda}=0$, hence $F_{\infty,z,\lambda}=0$ and the first property follows.

Then  we write
\[
I_+(q_{t,z-2t\lambda})
	=I_+(\Pi u_0)-2tI_+(T_{u_0}f_{t,z-2t\lambda})
	=I_+(\Pi u_0)-2\sqrt t\langle \sqrt t f_{t,z-2t\lambda},\Pi u_0\rangle,
\]
and the second property follows. The proof is complete.
\end{proof}

\section{Proof of Properties \eqref{loc} and \eqref{distorted}}\label{3.8}


\subsection{Proof of the localization Property~\eqref{loc}}

The first step will be to show that the contribution~\eqref{loc} from the integral away from  $\xi=\lambda$ vanishes as $t\to\infty$.


We recall that $\widehat{q}_{t,z-2t\lambda}/\sqrt{t}$ is bounded in $L^2_\lambda(\Lambda,H^1_\xi(0,\infty))$ thanks to Lemma~\ref{lem:q-bound}, in particular it is bounded in  $L^2_\lambda(\Lambda,L^{\infty}_\xi(0,\infty))$.
Recall the smooth cutoff function $\chi(\xi)$ defined relative to $\Lambda=[\Lambda_-,\Lambda_+]\subset (0,\infty)$ at the beginning of Section~\ref{sec:radiationlimit}.
By integration by parts, we have
\begin{multline}\label{eq:IPP0}
\sqrt{2t}\int_0^{\infty}\ee^{\ii t (\xi-\lambda)^2+\ii z \xi}\widehat{q}_{t,z-2t\lambda}(\xi)(1-\chi(\xi))\frac{\dd\xi}{2\pi}\\
	= -\ii\frac{\ee^{\ii t \lambda^2}I_+(q_{t,z-2t\lambda})}{ 2\pi \sqrt{2t}(\lambda-z/2t)}
	-\frac{\ii}{\sqrt{2t}}\int_0^{\infty}\ee^{\ii t (\xi-\lambda)^2+\ii z \xi}\frac{\dd}{\dd \xi}\left(\frac{ \widehat{q}_{t,z-2t\lambda}(\xi)(1-\chi(\xi))}{\xi-\lambda+z/2t}\right)\frac{\dd\xi}{2\pi}.
\end{multline}
The first term in the right-hand side goes to zero weakly in $L^2(\Lambda)$ as $t\to+\infty$ thanks to Lemma~\ref{lem:q-I+}. Concerning the second term in the right-hand side, we split the integral depending on where the derivative falls. For instance, we detail the upper bound on
\[
\frac{1}{\sqrt{2t}}\int_0^{\infty}\ee^{\ii t (\xi-\lambda)^2+\ii z \xi}\frac{( \widehat{q}_{t,z-2t\lambda})'(\xi)(1-\chi(\xi))}{\xi-\lambda+z/2t}\frac{\dd\xi}{2\pi}.
\]
We decompose $\widehat{q}_{t,z-2t\lambda}=\widehat{\Pi u_0}-2t \widehat{T_{u_0} f_{t,z-2t\lambda}}$. The term corresponding to $\widehat{\Pi u_0}$ tends to zero thanks to the Cauchy-Schwarz inequality
\[
\frac{1}{\sqrt{2t}}\int_0^{\infty}\ee^{\ii t (\xi-\lambda)^2+\ii z \xi}\frac{(\widehat{\Pi u_0})'(\xi)(1-\chi(\xi))}{\xi-\lambda+z/2t}\frac{\dd\xi}{2\pi}
	\leq \frac{1}{\sqrt{2t}}\left\|\frac{1-\chi(\xi)}{\xi-\lambda+z/2t}\right\|_{L^2_\xi}\|(\widehat{\Pi u_0})'\|_{L^2_\xi},
\]
where $\|(\widehat{\Pi u_0})'\|_{L^2_\xi}$ is finite since $xu_0\in L^2_x$ and $\left\|(1-\chi(\xi))/(\xi-\lambda+z/2t)\right\|_{L^2_\xi}$ is  finite since $\xi$ is located away from $\lambda\in\Lambda$ on the support of $1-\chi$. The term corresponding to $2t\widehat{T_{u_0}f_{t,z-2t\lambda}}$ can be written as
\begin{multline}\label{eq:IPP}
\int_0^{\infty}\ee^{\ii t (\xi-\lambda)^2+\ii z \xi}\frac{(\sqrt{2t}\widehat{T_{u_0} f_{t,z-2t\lambda}})'(\xi)(1-\chi(\xi))}{\xi-\lambda+z/2t}\frac{\dd\xi}{2\pi}\\
	=\int_0^{\infty}\ee^{\ii t (\xi-\lambda)^2+\ii z \xi}\frac{1-\chi(\xi)}{\xi-\lambda+z/2t}\int_0^{\infty}\sqrt{2t}\widehat{f}_{t,z-2t\lambda}(\eta)(\widehat{u_0})'(\xi-\eta)\dd \eta \frac{\dd\xi}{2\pi}.
\end{multline}
Thanks to the weak convergence of $x\mapsto \sqrt{2t}f_{t,z-2t\lambda}$ to zero in $L^2_{x,\lambda}(\R\times\Lambda)$, we know that pointwise in $\xi$, there holds $\int_0^{\infty}\sqrt{2t} \widehat{f}_{t,z-2t\lambda}(\eta)(\widehat{u_0})'(\xi-\eta)\dd \eta\rightharpoonup 0$ in $L^2(\Lambda)$ as $t\to+\infty$. Moreover, we have from the Young inequality
\begin{multline*}
\left|\frac{\ee^{\ii t (\xi-\lambda)^2+\ii z \xi}(1-\chi(\xi))}{\xi-\lambda+z/2t}\int_0^{\infty}\sqrt{2t}\widehat{f}_{t,z-2t\lambda}(\eta)(\widehat{u_0})'(\xi-\eta)\dd \eta	\right| \\
	\leq C\left|\frac{\ee^{\ii z \xi}(1-\chi(\xi))}{\xi-\lambda}\right| \|\sqrt{2t}\widehat{f}_{t,z-2t\lambda}\|_{L^2_\xi}\|(\widehat{u_0})'\|_{L^2_\xi},
\end{multline*}
where $\|\sqrt{2t}\widehat{f}_{t,z-2t\lambda}\|_{L^2}$ is bounded and $\left|\frac{\ee^{\ii z \xi}(1-\chi(\xi))}{\xi-\lambda}\right|\in L^1_\xi$ thanks to the fact that $\mathrm{Im}(z)>0$. We can therefore apply the dominated convergence theorem and get the weak $L^2(\Lambda)$ convergence: 
\[
\int_0^{\infty}\ee^{\ii t (\xi-\lambda)^2+\ii z \xi}\frac{(\sqrt{2t}\widehat{T_{u_0}f_{t,z-2t\lambda}})'(\xi)(1-\chi(\xi))}{\xi-\lambda+z/2t}\frac{\dd\xi}{2\pi}\tdw_t,\infty 0.
\]
In the second term of the right-hand side of~\eqref{eq:IPP0}, the other terms obtained when the derivative falls either on $1-\chi$ or on $\frac{1}{\xi-\lambda+z/2t}$ are treated similarly. 

\subsection{Proof of Property \eqref{distorted}}
In this paragraph, we focus on the second term \eqref{distorted}. From now on, we shall make an extensive use of stationary phase asymptotics, which we recall in the following lemma, see e.g. \cite{Stein93}.

\begin{lemma}[The van der Corput estimate and  stationary phase asymptotics]\label{lem:phasestat}
There exists a universal constant $\kappa_0$ such that, for every open interval $I$, for every $C^2$ function $\phi :I\to \R$ such that $\phi ''\geq 1$ on $I$, for every  continuous function $b\in L^1(I)$ such that $b'\in L^1(I)$, 
\begin{equation} \label{VdC}
\forall \tau \geq 1,\ \forall x_0\in I,\  \left |\int_I\ee^{\ii \tau \phi (x)}b(x)\dd x\right |\leq \kappa_0\tau^{-1/2}\left (|b(x_0)|+\int_I |b'(x)|\dd x\right ).
\end{equation}
Furthermore, if $\phi$ has a critical point $x_c$ in $I$ and if moreover $b\in H^1(I)$ is compactly supported, 
\begin{equation}\label{phasestat}
\int_I\ee^{\ii \tau \phi (x)}b(x)\dd x=\left (\frac{2\pi}{\tau \phi''(x_c)}\right )^{1/2}\ee^{\ii \pi/4}(b(x_c)+o(1)), 
\quad
\tau \to \infty.
\end{equation}
\end{lemma}


In the integral~\eqref{distorted},  we first recognize the distorted Fourier transform of $\Pi u_0$ in the integral over $x$:
\[
\sqrt{2t}\int_{0}^{\infty}\int_{\R}u_0(x)\ee^{\ii t(\xi-\lambda)^2+\ii z\xi } \chi (\xi)\overline{m}_-\left(x,\xi\right)\dd x \frac{\dd \xi}{2\pi}=\sqrt{2t}\int_{0}^{\infty} \ee^{\ii t(\xi-\lambda)^2+\ii z\xi} \chi(\xi)\widetilde{\Pi u_0}\left(\xi\right)\frac{\dd \xi}{2\pi}.
\]
In order to apply the stationary phase lemma, we use the fact that $\widetilde{\Pi u_0}\in H^1_\xi(\eps,\infty)$ for every $\eps>0$ (see Lemma~\ref{lem:distortedH1} below), and in particular it is continuous on $(0,\infty)$. 
Since $\chi$ is supported on $[\Lambda_-/4,4\Lambda_+]$, we can apply \eqref{phasestat} from Lemma \ref{lem:phasestat} and deduce that
\[
\sqrt{2t}\int_{0}^{\infty}\int_{\R}u_0(x)\ee^{\ii t(\xi-\lambda)^2+\ii z\xi } \chi (\xi)\overline{m}_-\left(x,\xi\right)\dd x \frac{\dd \xi}{2\pi}
	\td_t,{+\infty} \frac{\ee^{\ii\pi/4}}{\sqrt{2\pi}} \ee^{\ii \lambda z}\widetilde{\Pi u_0}(\lambda).
\]
The convergence is  clearly uniform in $\lambda\in\Lambda$.

\section{Proof of Property \eqref{cancel}}\label{3.9}

In this section, we prove the remainder term estimate~\eqref{cancel}, which turns out to be the most delicate one. 
Again recalling the cutoff function $\chi$ defined at the beginning of Section~\ref{sec:radiationlimit}, let us set
\[ p_{t,z,\lambda}(x)=\int_{\R} \ee ^{-\ii\frac{\alpha ^2}{2}-\ii\frac{\overline z\alpha }{\sqrt{2t}}}\chi\left(\lambda+\frac{\alpha}{\sqrt{2t}}\right)r_{t,z,\lambda}(x,\alpha )\dd\alpha .\]
 Our goal is to prove that the following quantity tends to $0$ weakly in $L^2(\Lambda )$ as $t\to \infty$,
 \begin{equation}\label{I}
 \la q_{t,z-2t\lambda}, p_{t,z,\lambda}\ra = I_1(t,z,\lambda )-I_2(t,z,\lambda),
 \end{equation}
where
\begin{align}\label{I_1}
 I_1(t,z, \lambda)&:=\la \Pi u_0, p_{t,z,\lambda}\ra ,\\
\label{I_2}
 I_2(t,z, \lambda)&:=2t\la T_{u_0}f_{t,z-2t\lambda}, p_{t,z,\lambda}\ra .
 \end{align}
\subsection{Decomposition of $p_{t,z,\lambda}$.}

 First, we derive a new expression for  $p_{t,z,\lambda}$.
 The starting point is the following observation, coming from the fact that  $m_-$ is a generalized eigenfunction of $L_{u_0}$:
 \[ T_{u_0}m_-(y,\eta)=\ee ^{\ii \eta y}Da(y,\eta ),\] 
 where we set $a(y,\eta ):= \ee^{-\ii\eta y}m_-(y,\eta).$ Then $a(y,\eta)\to 1$ as $y\to -\infty$.
 Plugging this identity into the definition \eqref{r} of $r_{t,z,\lambda}(x,\alpha )$, and integrating by parts in $y$, we obtain
 \begin{align*}
 r_{t,z,\lambda}(x,\alpha )&=\ii\int_{-\infty}^x \ee ^{\ii\lambda x}Da\left (y,\lambda +\frac{\alpha }{\sqrt{2t}}\right )\left ( \mathrm{e}^{-\ii\overline z\frac{(x-y)}{2t}+\ii\frac{(x^2-y^2)}{4t}+\ii\frac{\alpha y}{\sqrt{2t}}}- \mathrm{e}^{\ii\frac{\alpha x}{\sqrt{2t}}}     \right )\dd y\\
 &=\ii\int_{-\infty}^x \ee ^{\ii\lambda x}\left (a\left (y,\lambda +\frac{\alpha }{\sqrt{2t}}\right )-1\right )\mathrm{e}^{-\ii\overline z\frac{(x-y)}{2t}+\ii\frac{(x^2-y^2)}{4t}+\ii\frac{\alpha y}{\sqrt{2t}}} \left [\frac{-\overline z}{2t}+\frac{y}{2t}-\frac{\alpha}{\sqrt{2t}}\right ]\dd y.
 \end{align*}
  Plugging this identity into the definition of $p_{t,z,\lambda}(x)$, and integrating by parts in $\alpha $, we observe a cancellation leading to
\begin{equation}\label{eq:p}
 p_{t,z,\lambda}(x)= -\ee ^{ \ii\lambda x}\int_{\R} \int_{-\infty}^x \ee^{-\ii\frac{\alpha ^2}2-\ii\frac{\overline z\alpha}{\sqrt{2t}}-\ii \overline z\frac{(x-y)}{2t}+\ii\frac{(x^2-y^2)}{4t}+\ii\frac{\alpha y}{\sqrt{2t}}}
 b\left (y,\lambda +\frac{\alpha }{\sqrt{2t}}\right )\frac{\dd y\dd \alpha}{2t} .
 \end{equation}
 where 
 \begin{equation}\label{def:b}
 b(y,\eta):=\chi '(\eta)(a(y,\eta)-1)+\chi (\eta)\partial_\eta a(y,\eta).
 \end{equation} 
Identity \eqref{eq:p} is a crucial step in our proof. According to Lemma~\ref{lem:a-structure} below, we isolate the term with no decay as $y\to-\infty$ in the expression of $b$, so that we  define
\begin{equation}\label{b-decompose}
b_0(y,\eta):=-\frac{1}{2\pi}\frac{\ee^{-\ii\eta y}}{\eta}\overline{\widetilde{\Pi u_0}}(\eta)\chi (\eta)
 	\quad  \text{ and }\quad
b_1(y,\eta):=b(y,\eta)-b_0(y,\eta).
\end{equation}
 
The part of~\eqref{eq:p} corresponding to $b_0$ decouples the variables $y$ and $\alpha$, so that we define 
\begin{align}
\label{eq:J0} J^0_{t,z,\lambda}
	&:=\int_\R \ee^{-\ii \frac{\alpha^2}2-\ii \frac{\overline z\alpha }{\sqrt{2t}}} \left.\frac{\overline{\widetilde{\Pi u_0}}(\eta)}{\eta}\chi (\eta)\right|_{\eta =\lambda +\frac{\alpha}{\sqrt{2t}}}
	\frac{\dd\alpha}{2\pi},
	\\
\label{eq:G0}  G^0_{t,z,\lambda}(x)
	&:=\int_{-\infty}^x \ee^{-\ii \overline z\frac{(x-y)}{2t}-\ii\frac{y^2}{4t}-\ii\lambda y}\frac{\dd y}{\sqrt{2t}}.
\end{align}
Finally, we denote the part corresponding to $b_1$ by
\begin{equation}\label{eq:P1-def} P^1_{t,z,\lambda}(x)
 	:= \int_\R \int_{-\infty}^x \ee^{-\ii \Phi_{t,z}(\alpha,x,y)}b_1\left (y,\lambda+\frac{\alpha}{\sqrt{2t}}\right ) \frac{\dd y\dd \alpha} {\sqrt{2t}},
\end{equation}
where for further reference, we introduce the phase function in the integral in the right hand side of \eqref{eq:p}:
 \begin{equation}\label{Phi,def} 
 \Phi _{t,z}(\alpha,x,y):=\frac{\alpha ^2}2+\frac{\overline z\alpha}{\sqrt{2t}}+\frac{\overline z(x-y)}{2t}-\frac{x^2-y^2}{4t}-\frac{\alpha y}{\sqrt{2t}}.
 \end{equation}

We have proven the following proposition.
\begin{proposition}[Decomposition of $p_{t,z,\lambda}$] We have a decomposition
 \begin{align}
 p_{t,z,\lambda}(x)&= p^0_{t,z,\lambda}(x)\ee^{\ii t \lambda^2}+ p^1_{t,z,\lambda}(x), \label{eq:p-decompo}\\
 \label{eq:p0} p^0_{t,z,\lambda}(x)&= \ee^{-\ii t\lambda^2+\ii \lambda x+\ii \frac{x^2}{4t}}J^0_{t,z,\lambda} G^0_{t,z,\lambda}(x)/\sqrt{2t},\\
 \label{eq:p1}  p^1_{t,z,\lambda}(x)&= -  \ee^{\ii \lambda x} P^1_{t,z,\lambda}(x)/\sqrt{2t}.                    
 \end{align}
\end{proposition}

The integrals $J^0_{t,z,\lambda}$ and $G^0_{t,z,\lambda}$ enjoy the following properties.
\begin{lemma}[Properties of $J^0_{t,z,\lambda}$ and $G^0_{t,z,\lambda}$]\label{lem:J-G}
The integral $J^0_{t,z,\lambda}$ is uniformly bounded in the variable $\lambda\in\Lambda$ as $t\to+\infty$, and convergent to 
\[ J^0_{t,z,\lambda}\td_t,\infty \ee^{-\ii \pi/4}\frac{\overline{\widetilde{\Pi u_0}}(\lambda)}{\sqrt{2\pi} \lambda}.\]
The function $G^0_{t,z,\lambda}$ is uniformly bounded in the variables $x\in\R$ and $\lambda\in\Lambda$, moreover, we have the pointwise convergence result
\begin{equation}\label{G0-lim} \ee^{-\ii t\lambda^2}G^0_{t,z,\lambda}(x)\td_t,\infty \sqrt{2\pi}\ee^{-\ii \pi/4} \ee^{\ii \overline z \lambda}.
\end{equation}
\end{lemma}
 
\begin{proof}
Since $\widetilde{\Pi u_0}$ belongs to $H^1([\Lambda_-/4,\infty ))$,  the stationary phase formula \eqref{phasestat}  leads to the result for $J^0_{t,z,\lambda}$.

We now focus on $G^0_{t,z,\lambda}$. For fixed $x\in \R$, we set $y=2t s$ in the $y$ integral and we get
\begin{equation}\label{eq:G}
G^0_{t,z,\lambda}(x)=\sqrt{2t}\ee^{-\ii \frac{\overline z x}{2t}}\int_{-\infty}^{\frac{x}{2t}} \ee^{\ii \overline zs-\ii ts^2-2\ii t\lambda s}\dd s, \end{equation}
and, after localization near the critical point $s=-\lambda$,  the stationary phase formula \eqref{phasestat} yields~\eqref{G0-lim}.
Then, writing $\sigma =\sqrt {2t}(s+\lambda)$ in \eqref{eq:G}, we have
\[ G^0_{t,z,\lambda}(x)=\ee^{\ii t\lambda ^2}\int_{-\infty}^{\sqrt{2t}(\lambda +\frac{x}{2t})}\ee^{\ii \overline z(\frac{\sigma}{\sqrt{2t}}-\frac{x}{2t}-\lambda)}\ee  ^{-\ii \frac{\sigma ^2}{2}} \dd \sigma ,\]
and the boundedness  follows from integration by parts.
\end{proof}

\subsection{Study of $I_2$.}
 Next we study the weak limit  in $L^2(\Lambda)$ as $t\to +\infty $ of 
 \[ I_2(t,z,\lambda)=\la \sqrt{2t} f_{t,z-2t\lambda}, \sqrt{2t}u_0p_{t,z,\lambda}\ra.\]
We set 
 \[ g_{t,z,\lambda}(x):=\sqrt{2t}\, u_0(x)\, p_{t,z,\lambda}(x).\]
 In view of the first statement of  Lemma \ref{lem:q-I+}, the following proposition implies that this limit is~$0$.
\begin{proposition}\label{prop:g-structure}
The two families 
\[
g^0_{t,z,\lambda}(x)=\sqrt{2t}u_0(x) p^0_{t,z,\lambda}(x)
\quad\text{ and }\quad
g^1_{t,z,\lambda}(x)=\sqrt{2t}u_0(x) p^1_{t,z,\lambda}(x)
\]
belong to a compact subset of $L^2_{\lambda,x}(\Lambda\times \R)$, and
\[ g_{t,z,\lambda}(x)=g^0_{t,z,\lambda}(x)\ee^{\ii t\lambda^2}+g^1_{t,z,\lambda}(x).\]
\end{proposition}
The next two paragraphs are devoted to the proof of Proposition \ref{prop:g-structure}: we show that  $g^0_{t,z,\lambda}$ and $g^1_{t,z,\lambda}$ belong to a compact subset of $L^2_{\lambda,x}(\Lambda\times \R)$.


\subsubsection{Study of $g^0_{t,z, \lambda}$.}
We have
\[g^0_{t,z,\lambda}(x)
	=\ee^{-\ii t\lambda^2+\ii \lambda x+\ii \frac{x^2}{4t}}u_0(x) J^0_{t,z,\lambda}G^0_{t,z,\lambda}(x).\]
Thanks to Lemma~\ref{lem:J-G}, we know that
\[ g^0_{t,z,\lambda}(x)\td_t,\infty -\ii\ee^{\ii \lambda x}u_0(x)\frac{\overline{\widetilde{\Pi u_0}}(\lambda)}{\lambda}.\]
Moreover, $J^0_{t,z,\lambda}$ and $G^0_{t,z,\lambda}(x)$ are uniformly bounded, hence there exists $g^*\in L^2(\R)$ such that
\[ \forall t\in \R, \forall (x,\lambda) \in \R\times \Lambda,\quad
 |g^0_{t,z,\lambda}(x)|\leq g^*(x).\]
In order to conclude to compactness of $(x,\lambda)\mapsto g^0_{t,z, \lambda}(x)$ in $L^2_{x,\lambda}( \R\times\Lambda)$, it only remains to apply the Lebesgue theorem.

 \subsubsection{Study of $g^1_{t,z,\lambda}$.}
 In this paragraph, we prove that the functions $$(x,\lambda)\mapsto g^1_{t,z,\lambda}(x)=-u_0(x)\ee^{\ii\lambda x}P^1_{t,z,\lambda}(x)$$ belong to a compact subset of $L^2(\R\times \Lambda)$ as $t\to \infty$. As a first step, we deal with the contribution of $x\leq -R$ for $R>\!\!>1$. 
 \begin{lemma}\label{lem:x<-R}
 There holds
\[ \sup_{x\leq -R}\limsup_{t\to \infty} \left | P^1_{t,z,\lambda}(x)  \right | =O(R^{-1/2}), \quad R\to +\infty.\]
 \end{lemma}
 \begin{proof} We recall that $P^1_{t,z,\lambda}$ was defined in~\eqref{eq:P1-def}. According to Lemma~\ref{lem:a-structure} below, the contribution from $b_1$ (defined in~\eqref{b-decompose}) as $y\to-\infty$  is of the form 
\[
b_1(y,\eta)
	=-\ii\chi'(\eta)\frac{c_0}{y+\ii}-\kappa(\eta) \frac{\ee^{-\ii\eta y}}{y+\ii}+O\left(\frac{1}{|y|^{3/2}}\right)
\]
where
\begin{equation}\label{def:kappa}  
 \kappa (\eta):=\ii\chi '(\eta)\frac{\overline{\widetilde{\Pi u_0}}(\eta)}{2\pi \eta}+\chi(\eta)\frac{\int_\R u_0(s)n(s,\eta)\dd s}{2\pi \eta}.
 \end{equation}
The contribution from $\kappa(\eta) \frac{\ee^{-\ii\eta y}}{y+\ii}$ decouples the variables of integration $y$  and $\alpha$, so that we set 
 \begin{equation}\label{def:G1}
 G^1_{t,z,\lambda}(x):=\int_{-\infty}^x \ee^{\ii \frac{\overline z}{2t}(y-x)-\ii\frac{y^2}{4t}-\ii\lambda y}\frac{\dd y}{(y+\ii)}.
 \end{equation}
For the contribution of $-\ii\chi'(\eta)\frac{c_0}{y+\ii}$, we write $\alpha :=\beta +\frac{y}{\sqrt{2t}}$ and use the identity
 \begin{equation}\label{eq:Phi,decompose}
 \Phi _{t,z}(\alpha, x,y)=\frac{1}{2}\left (\alpha -\frac{y}{\sqrt{2t}}\right )^2+\frac{\overline z}{\sqrt{2t}}\left (\alpha -\frac{y}{\sqrt{2t}}\right )+\frac{\overline z x}{2t}-\frac{x^2}{4t}.
 \end{equation}
Hence we find that as $x\to-\infty$,
    \begin{multline}\label{eq:P1}
P^1_{t,z,\lambda}(x)=-\ii c_0   \ee^{\ii \frac{x^2}{4t}-\ii \frac{\overline z x}{2t}}\int_{-\infty}^x \left [\int_\R \ee^{-\ii \frac{\beta^2}{2}-\ii \frac{\overline z\beta}{\sqrt{2t}}}\chi '\left (\lambda+\frac{\beta}{\sqrt{2t}}+\frac{y}{2t}\right )d\beta\right ] \frac{\dd y}{\sqrt{2t}(y+\ii )}\\
- \ee^{\ii \frac{x^2}{4t}}G^1_{t,z,\lambda}(x)\int_\R \ee^{-\ii \frac{\alpha^2}{2}-\ii \frac{\overline z\alpha}{\sqrt{2t}}}\kappa \left (\lambda+\frac{\alpha}{\sqrt{2t}}\right )\, \frac{\dd \alpha}{\sqrt{2t}} +O\left (\frac{1}{|x|^{1/2}}\right ).
 \end{multline} 
 Denote by $P^{1,1}_{t,z,\lambda}(x)$ and $P^{1,2}_{t,z,\lambda}(x)$ the two terms of the right--hand side of \eqref{eq:P1}.

 Integrating by parts in $\beta$, and observing that 
 \[ \left |\frac{\beta}{\sqrt{2t}}+\frac{y}{2t}\right |\lesssim 1,\]
 we obtain
 \begin{equation}\label{est:P11}
 |P^{1,1}_{t,z,\lambda}(x)| \lesssim \int_{-\infty}^x \ee^{\mathrm{Im}z\frac{(y-x)}{2t}}\frac{\dd y}{\sqrt{2t}(1+|y|)}=\frac{O(\log |x|+\log |t|)}{\sqrt t}.
 \end{equation}
 Next we estimate $P^{1,2}_{t,z,\lambda}(x)$. We claim that
  \begin{equation}\label{est:P12}
 |P^{1,2}_{t,z,\lambda}(x)| =O\left (  \frac1{|x|}+\frac{1}{\sqrt t}  \right ).
 \end{equation}
Indeed, since $\kappa $ is a bounded compactly supported function, the integral in $\alpha $ is bounded, and we just have to estimate $ G^1_{t,z,\lambda}(x)$. 
 We decompose 
 \[G^1_{t,z,\lambda}(x)=\int_{-\infty}^x\ee^{\ii \frac{\overline z}{2t}(y-x)-\ii\frac{y^2}{4t}-\ii\lambda y}\left(1-\chi\left (\frac{-y}{2t}\right )\right)\frac{\dd y}{(y+\ii )}+\int_{-\infty}^x\ee^{\ii \frac{\overline z}{2t}(y-x)-\ii\frac{y^2}{4t}-\ii\lambda y}\chi\left (\frac{-y}{2t}\right )\frac{\dd y}{(y+\ii )}.\] 
 In the first integral, the derivative $\lambda +y/2t-\overline z/2t$ of the phase function is away from $0$, therefore an integration by parts shows that these integrals are 
 $O(|x|^{-1}+t^{-1})$.
  In the second integral, setting $y=2t s$ and applying \eqref{VdC} leads to a bound $O(t^{-1/2})$.
 
 Finally, estimates \eqref{eq:P1}, \eqref{est:P11} and \eqref{est:P12} complete the proof of Lemma \ref{lem:x<-R}.
 \end{proof}
\begin{remark}\label{rk:x>-R}
Using Lemma \ref{lem:x<-R}, we infer
\[ \limsup_{t\to \infty}\| g^1_{t,z,\lambda}\|_{L^2(x<-R)}= O(R^{-1/2}).\]
Taking $x=-R$ in Lemma \ref{lem:x<-R} and using that \[\mathrm{Im}(\Phi_{t,z}(\alpha,x,y)-\Phi_{t,z}(\alpha,-R,y))=-\mathrm{Im}z(x+R)\leq 0 ,\] we also infer that
\[ \sup_{x\geq -R} \limsup_{t\to \infty} \left | \int_\R \int_{-\infty}^{-R} \ee^{\ii \Phi_{t,z}(\alpha,x,y)}b_1\left (y,\lambda+\frac{\alpha}{\sqrt{2t}}\right )\, \frac{\dd y\dd \alpha} {\sqrt{2t}}  \right | =O(R^{-1/2}), \quad R\to +\infty.\]
\end{remark}

\begin{lemma}\label{lem:x>-R} We set, for $R>0$, and $x\geq -R$, 
\[ g^1_{t,z,\lambda,R}(x):=- {\bf 1}_{x\geq -R} u_0(x)\ee ^{ \ii\lambda x}\int_{-\infty}^\infty \int_{-R}^x \ee^{-\ii\Phi _{t,z}(\alpha,x,y)}
 b_1\left (y,\lambda +\frac{\alpha }{\sqrt{2t}}\right )\frac{\dd y\dd \alpha}{2t}.\]
Then for every $R>0$, the family of functions $(x,\lambda)\mapsto g^1_{t,z,\lambda,R}(x)$ belongs to a compact subset of $L^2([-R,\infty)\times \Lambda)$ as $t\to \infty$.
\end{lemma}
\begin{proof}
 For $x\geq -R$, let us rewrite $g^1_{t,z,\lambda,R}(x)$ as
\[ g^1_{t,z,\lambda,R}(x)=-\ee ^{\ii\lambda x}\ee ^{\ii \frac{\overline zx}{2t}+\ii \frac{x^2}{4t}}u_0(x)\int_{-R}^x G^2_{t,z,\lambda}(y)\frac{\dd y}{\sqrt{2t}},\]
with
\begin{equation}\label{eq:G2} G^2_{t,z,\lambda}(y):=\int_\R \ee^{-\ii\frac{\beta^2}{2}-\ii\frac{\overline z \beta}{\sqrt{2t}} }b_1\left (y, \lambda+\frac{\beta}{\sqrt{2t}}+\frac{y}{2t}  \right )\dd \beta.
\end{equation}
Since $b_1$ is uniformly bounded thanks to Lemma~\ref{lem:a-structure}, and since the domain of the integral in $\beta$ has a length of order $\sqrt{2t}$, we infer that
$G^2_{t,z,\lambda}(y)/\sqrt{2t}$ is uniformly bounded:
\begin{equation}\label{eq:G2-bound}
|G^2_{t,z,\lambda}(y)/\sqrt{2t}|	\leq C\ee^{(\mathrm{Im}z) y/(2t)}.
\end{equation}
 As a consequence, using $y\leq x$ and $|\ee ^{\ii \overline zx/2t}|\leq \ee^{-(\mathrm{Im}z) x/(2t)}$, a crude estimate yields
\[|g^1_{t,z,\lambda,R}(x)|\leq C |x+R||u_0(x)|.\]
 Furthermore, recall that the function $(y,\eta)\mapsto b_1(y,\eta)$ is continuous on $\R \times (0,\infty )$.  Hence the function $(x,\lambda )\mapsto \int_{-R}^x G_{t,z,\lambda}(y)\frac{\dd y}{\sqrt{2t}}$ is uniformly equicontinuous on $[-R,L]\times \Lambda$ and, up to extracting a subsequence in  $t$, $g^1_{t,z,\lambda,R}(x)$ is pointwise convergent in $x,\lambda$. 
By the dominated convergence theorem, we infer that $g^1_{t,z,\lambda,R}$ belongs to a compact subset of $L^2([-R,\infty)\times \Lambda )$.
\end{proof}
Combining Lemma~\ref{lem:x<-R}, Remark~\ref{rk:x>-R} and Lemma~\ref{lem:x>-R}, we have shown that $g^1_{t,z,\lambda}$ belongs to a compact subset of $L^2(\R\times \Lambda )$.
 This completes the proof of Proposition \ref{prop:g-structure}.

\subsection{Study of $I_1$.} In comparison to $I_2(t,z,\lambda)$ studied in the previous paragraph, $I_1(t,z,\lambda)$ has an extra factor $t^{-1/2}$ but the decay of the factor $\Pi u_0$ is slower, since we can only have the decomposition
\begin{equation}\label{Piu, decompose}
\Pi u_0(x)=\frac{\Pi ((\diamond +\ii)u_0)(x)}{x+\ii}+ \frac{\ii\int_\R u_0(s)\dd s}{2\pi(x+\ii)}= v_1(x)+\frac{c_1}{x+\ii},
\end{equation}
where $v_1\in L^1(\R)$. We recall that $I_1(t,z,\lambda)=\la \Pi u_0,p_{t,z,\lambda}\ra $ where $p_{t,z,\lambda}$ is also decomposed into two parts according to~\eqref{eq:p-decompo}. 

 \subsubsection{The contribution of $p^0_{t,z,\lambda}$.} 
We already saw in Lemma~\ref{lem:J-G} that $J^0_{t,z,\lambda}$ is uniformly bounded, and that $G^0_{t,z,\lambda}$ is uniformly bounded and pointwise convergent.
Hence, by \eqref{eq:p0} and the dominated convergence theorem, there holds the strong convergence in $L^2(\Lambda)$
\[ \la v_1, p^0_{t,z,\lambda}\ra \td_t,\infty 0.\]
\\
Let us come to the study of $\la (x+\ii)^{-1},p^0_{t,z,\lambda}\ra $. We have
\[\la (x+\ii)^{-1},p^0_{t,z,\lambda}\ra = \overline{J^0_{t,z,\lambda}}\ee^{\ii t\lambda^2}\int_\R \frac{\ee^{-\ii \lambda x-\ii \frac{x^2}{4t}}}{x+\ii }\overline{ G^0_{t,z,\lambda}(x)}\frac{\dd x}{\sqrt{2t}}.\]
 Again, we extract the contribution of the critical point $x=-2t\lambda$ by introducing the cutoff function $\chi (-x/2t)$. In the integral  containing $1-\chi(-x/2t)$, we integrate once by parts in $x$, and we observe that this integral tends to $0$ as $t\to \infty$. As for the integral containing $\chi(-x/2t)$, we set $y=2ts$ in the definition~\eqref{eq:G0} of $G^0_{t,z,\lambda}$ and $x=-2t\mu$: we observe that this integral is bounded by 
\[ O(1) \sup_{\Lambda_-/4 \leq \mu \leq 4\Lambda_+}\left |\int_{-\infty}^{-\mu} \ee^{-\ii z(\mu+s)+\ii t(s^2+2\lambda s)}\dd s\right |=O(t^{-1/2}),\]
applying again the van der Corput estimate \eqref{VdC}.

Summing up, we have proved that $\la \Pi u_0,p^0_{t,z,\lambda}\ra \td_t,\infty 0$ strongly in $L^2(\Lambda)$.

 \subsubsection{The contribution of $p^1_{t,z,\lambda}$.} 
We first study the most delicate contribution in the decomposition~\eqref{Piu, decompose} of $\Pi u_0$, which is the one of $c_1/(x+\ii)$. We have, using again \eqref{eq:Phi,decompose} in the expression~\eqref{eq:P1-def} of $P^1_{t,z,\lambda}$,
\begin{equation}\label{eq:p1/(x+i)-2}
\la (x+\ii)^{-1},p^1_{t,z,\lambda}\ra = - \int_\R \frac{\ee^{-\ii \lambda x-\ii \frac{x^2}{4t}+\ii \frac{zx}{2t}}}{x+\ii}\int_{-\infty}^x \overline{G^2_{t,z,\lambda}}(y)\, \frac{\dd y\dd x}{{2t}},
\end{equation}
where $G^2_{t,z,\lambda}$ was defined in~\eqref{eq:G2}. Let us improve the bound~\eqref{eq:G2-bound}. 
Integrating by parts in $\beta$, and observing that, on the domain of the integral defining $G^2_{t,z,\lambda}$, 
$
\left |\frac{\beta}{\sqrt{2t}}+\frac{y}{2t}\right |\lesssim 1,
$
so that we get
\begin{equation}\label{est:B}
 |G^2_{t,z,\lambda}(y)|\lesssim \ee^{\frac{y\mathrm{Im}z}{2t}},
 \end{equation}
 and consequently
\begin{equation}\label{est:intB}
\left |\ee^{ \ii \frac{zx}{2t}} \int_{-\infty}^x \overline{G^2_{t,z,\lambda}}(y)\, \frac{\dd y}{{2t}} \right |=O(1).
\end{equation}

Next, as in the previous paragraph, we extract the contribution of the critical point $x=-2t\lambda$ by introducing the cutoff function $\chi(-x/2t)$. Then in the integral  containing $1-\chi(-x/2t)$, we integrate  by parts in $x$.  We get
\begin{equation}\label{eq:IPP3}
\ii\int_\R  \frac{\ee^{-\ii \lambda x-\ii \frac{x^2}{4t}+\ii \frac{zx}{2t}}}{x+\ii}\int_{-\infty}^x \overline{G^2_{t,z,\lambda}}(y) \frac{\dd y\dd x}{2t}=J_1+J_2+J_3+J_4,
\end{equation}
\begin{align*}
J_1
	&:=\int_\R \left (1-\chi\left (-\frac{x}{2t}\right )\right )\ee^{-\ii \lambda x-\ii \frac{x^2}{4t}+\ii \frac{zx}{2t}}\frac{\dd}{\dd x}\left (\frac{1}{(x+\ii)\left (\lambda +\frac{x-z}{2t}\right)}\right )\int_{-\infty}^x \overline{G^2_{t,z,\lambda}}(y) \frac{\dd y\dd x}{2t},\\
J_2
	&:=\int_\R \chi'\left (-\frac{x}{2t}\right )\frac{\ee^{-\ii \lambda x-\ii \frac{x^2}{4t}+\ii \frac{zx}{2t}}}{(x+\ii)\left (\lambda +\frac{x-z}{2t}\right)}\int_{-\infty}^x \overline{G^2_{t,z,\lambda}}(y) \frac{\dd y\dd x}{(2t)^2},\\
J_3
	&:=\int_\R \left (1-\chi\left (-\frac{x}{2t}\right )\right )\frac{\ee^{-\ii \lambda x-\ii \frac{x^2}{4t}+\ii \frac{zx}{2t}}}{(x+\ii)\left (\lambda +\frac{x-z}{2t}\right )} \overline{G^2_{t,z,\lambda}}(x)\frac{\dd x}{2t},
\\
J_4
	&:=\ii\int_\R \chi\left (-\frac{x}{2t}\right )\frac{\ee^{-\ii \lambda x-\ii \frac{x^2}{4t}+\ii \frac{zx}{2t}}}{x+\ii }\int_{-\infty}^x\overline{G^2_{t,z,\lambda}}(y)\frac{\dd y\dd x}{2t}.
\end{align*}
{\noindent\it Study of $J_1$ and $J_2$.} In view of \eqref{est:intB} and of the inequality $|x|\lesssim |x+2\lambda t|$ on the support of the integrand, the integrand in $J_1$ is uniformly  bounded by $(1+x^2)^{-1}$. Furthermore,  recalling $\Lambda=[\Lambda_-,\Lambda_+]\subset(0,\infty)$, one can use the uniform estimate in $\eta\in[\Lambda_-/4,4\Lambda_+]$, as $y\to -\infty $,
\[ |b_1(y,\eta)|\lesssim \frac{1}{1+|y|},\]
 from Lemma \ref{lem:a-structure} to improve \eqref{est:B} into
 \begin{equation}\label{est:Bminus}
  |{G^2_{t,z,\lambda}}(y)|\lesssim (1+|y|)^{-1}\ee^{\frac{y\mathrm{Im}z}{2t}},
 \end{equation}
as $y\to -\infty$. We conclude that, as $t\to \infty $, for fixed $x\in \R$,
\begin{equation}\label{est:intBminus}
 \left |\ee^{ \ii \frac{zx}{2t}} \int_{-\infty}^x \overline{G^2_{t,z,\lambda}}(y)\, \frac{\dd y}{2t} \right |=O\left ( \frac{\log t}{t}   \right ).
\end{equation}
Hence, by dominated convergence, $J_1\to 0$ uniformly in $\lambda\in\Lambda$ as $t\to \infty$. A similar argument leads to $J_2\to 0$, since, on the support of the integrand, $t\simeq |x|$.

{\noindent\it Study of $J_3$.} As for the third integral $J_3$, we perform again a partial integration,
\begin{align*}
\ii J_3
	=& \int_\R \left (1-\chi\left (-\frac{x}{2t}\right )\right )\ee^{-\ii \lambda x-\ii \frac{x^2}{4t}+\ii \frac{zx}{2t}}\frac{\dd }{\dd x}\left (\frac 1{(x+\ii )\left (\lambda +\frac{x-z}{2t}\right )^2} \right )\overline{G^2_{t,z,\lambda}}(x) \frac{\dd x}{2t}\\
	&+ \int_\R \chi'\left (-\frac{x}{2t}\right )\frac{\ee^{-\ii \lambda x-\ii \frac{x^2}{4t}+\ii \frac{zx}{2t}}}{(x+\ii )\left (\lambda +\frac{x-z}{2t}\right )^2} \overline{G^2_{t,z,\lambda}}(x)\frac{\dd x}{(2t)^2}\\
	&+\int_\R \left (1-\chi\left (-\frac{x}{2t}\right )\right )\frac{\ee^{-\ii \lambda x-\ii \frac{x^2}{4t}+\ii \frac{zx}{2t}}}{(x+\ii )\left (\lambda +\frac{x-z}{2t}\right )^2} \partial_x \overline{G^2_{t,z,\lambda}}(x)\frac{\dd x}{2t}.
\end{align*}
Using \eqref{est:B}, the first two integrals are $O(t^{-1})$. 

For the third integral, we use finer properties of $b_1$, that we decompose as
\begin{equation}\label{eq:b1,decompose}
 b_1(y,\eta)=(a(y,\eta)-1)\chi'(\eta)+ a^\sharp (y,\eta)\chi(\eta),
 \end{equation}
where $a^\sharp$ is defined in~\eqref{eq:asharp}. Lemmas \ref{lem:a-structure} and \ref{lem:asharp-derivative}, then \eqref{eq:a-1} combined with \eqref{Tum}  imply that 
\begin{equation}\label{b1-derivative}
\Vert \partial_yb_1 (\diamond ,\eta)\Vert_{\mathscr F}+\Vert \partial_\eta b_1 (\diamond,\eta)\Vert_{L^\infty(\R)}\leq C,
\end{equation}
where $\mathscr F$ is the space of functions $f\in L^1_{\mathrm{loc}}(\R)$ such that 
\[f(x)=f_1(x)+f_2(x),\]
where $f_2\in L^1(\R)$ and $(1+|x|)f_1(x)$ is bounded on $\R$, equipped with the norm \[\Vert f\Vert_\mathscr F:=\inf_{f=f_1+f_2}\Vert (1+|\diamond|)f_1\Vert_{L^\infty (\R)}+\Vert f_2\Vert_{L^1(\R)}.\]
We deduce the estimate
\[
\ee^{ \ii \frac{zx}{2t}}\partial_x\overline{G^2_{t,z,\lambda}}(x)=R_1(x)+R_2(x),\ 
\Vert R_1\Vert _\mathscr F\lesssim \sqrt t,\ \Vert R_2\Vert_{L^\infty}\lesssim \frac{1}{\sqrt{2t}}, \]
which follows from the estimates \eqref{b1-derivative} and of a crude estimate of the integral in $\beta$, since the support of integrand has a length $\lesssim \sqrt t$. Using again that $|x|\lesssim |x+2\lambda t|$ on the support of the integrand,  we  conclude that this integral is $O(t^{-1/2})$.

{\noindent\it Study of $J_4$.} Finally, let us study the integral $J_4$ involving the localization $\chi(-x/2t)$.

We set $x=-2t\mu $ and $y=2ts$, so that
\[ J_4=\int_\R \chi(\mu )\frac{2t }{2\ii t\mu +1}\ee^{2\ii t\lambda \mu -\ii t\mu ^2-\ii z\mu}\mathcal B_{t,z}(\mu)\dd \mu,\]
with
\[ \mathcal B_{t,z}(\mu):=\int_{-\infty}^{-\mu}\overline{G^2_{t,z,\lambda}}(2ts)\dd s.\]
In view of  the support of $\chi $, the support of the integrand is away from $\mu =0$, and the amplitude $\frac{2t }{2\ii t\mu +1}$ is bounded with all its derivatives.
Moreover, in view of the bounds \eqref{est:B} and \eqref{est:intB}, we have uniformly in the support of $\chi$,
\[ |\mathcal B_{t,z}(\mu)| =O(1),\  |\partial_\mu \mathcal B_{t,z}(\mu)| =O(1).\]
Applying the van der Corput estimate~\eqref{VdC}, we infer 
\[ J_4=O(t^{-1/2}).\]
Summing up, we have proved that $\la (x+\ii)^{-1},p^1_{t,z,\lambda}\ra \td_t,\infty 0$, uniformly in $\lambda \in \Lambda$. 

The contribution of the  other term 
$ \la v_1, p^1_{t,z,\lambda}\ra $
coming from \eqref{Piu, decompose} is easier to handle because $v_1$ is integrable, so, from \eqref{est:intB}, the integrand is then bounded by a fixed integrable function. Furthermore, we have already noticed in \eqref{est:intBminus} that the integrand is pointwise convergent to $0$. Therefore we can conclude using the dominated convergence theorem.

\appendix

\section{Properties of the distorted Fourier transform}\label{sec:distorted}

In this section, we derive some properties of the distorted Fourier transform and of the function $a$ defined as $a(x,\lambda)=\ee^{-\ii\lambda x}m_-(x,\lambda)$.

\subsection{Regularity of the distorted Fourier transform}

\begin{lemma}\label{lem:distortedH1}
Let $\eps>0$. Then $\widetilde{\Pi u_0}\in H^1(\eps,\infty)$.
\end{lemma}

\begin{proof}
We already know that $\widetilde{\Pi u_0}\in L^2(0,\infty)$, hence we focus on the derivative
\[
\widetilde{\Pi u_0}'(\lambda)
	=\int_{\R}u_0(x)\overline{\partial_\lambda m_-}(x,\lambda)\dd x.
\]
We make the change of functions  $m_-(x,\lambda)=\ee^{\ii\lambda x}a(x,\lambda)$ and write
\[
\widetilde{\Pi u_0}'(\lambda)
	=-\int_{\R}\ii xu_0(x)\ee^{-\ii\lambda x}  \overline{m_-}(x,\lambda)\dd x
	+\int_{\R}u_0(x) \ee^{-\ii\lambda x} \overline{\partial_\lambda a}(x,\lambda)\dd x.
\]
Since $x u_0\in L^2_x(\R)$, the first term in the right-hand side is the distorted Fourier transform of a $L^2$ function, hence it belongs to $L^2_\lambda(0,\infty)$. Concerning the second term in the right-hand side, since $u_0\in L^1_x(\R)$, it is enough to show that 
\[
\|\partial_\lambda a(\cdot,\lambda)\|_{L^{\infty}_x}\in L^2_\lambda(\eps,\infty).
\]

We use the equation~\eqref{m} satisfied by $m_-$ to get that $a$ is solution to
 \begin{equation}\label{eq:a-1}
  a(y,\lambda)-1=\ii \int_{-\infty}^y \ee ^{-\ii \lambda x}T_{u_0}m_-(x,\lambda)\dd x .
  \end{equation}
We write $a(x,\cdot)$ as the limit in $L^{\infty}(\Lambda)$ as $\delta\to 0$ of the following absolutely convergent integral,
\[
a(x,\lambda)-1
	=\lim_{\delta\to 0}
	\ii \int_{-\infty}^x\ee^{-\ii \lambda y}T_{u_0}(\ee^{\ii\lambda\diamond}a(\diamond,\lambda))(y)\frac{\dd y}{(1-\ii\delta y)^2}.
\]
Then $\partial_\lambda a$ is the limit in the sense of distributions of
\begin{align}\label{eq:a-derivative}
\partial_\lambda a(x,\lambda)
	=&-\lim_{\delta\to 0}
	 \int_{-\infty}^x\ee^{-\ii \lambda y}[T_{u_0},y](\ee^{\ii\lambda\diamond}a(\diamond,\lambda))(y))\frac{\dd y}{(1-\ii\delta y)^2}
	\\ \nonumber
	&+\lim_{\delta\to 0}
	\ii \int_{-\infty}^x\ee^{-\ii \lambda y}T_{u_0}(\ee^{\ii\lambda\diamond}\partial_\lambda a(\diamond,\lambda))(y)\frac{\dd y}{(1-\ii\delta y)^2}.
\end{align}
Now we note that when $T_{u_0}f\in\mathrm{Dom}(X^*)$, there holds
\begin{equation}\label{eq:X*}
X^* T_{u_0}f(y) =T_{u_0}(yf)
\quad
\text{and}
\quad
X^* T_{u_0}f (y)= y T_{u_0}f (y) +\frac{I_+(T_{u_0} f)}{2\ii\pi}.
\end{equation}
This formula is valid when $f=\ee^{\ii\lambda\diamond}a(\diamond,\lambda)$ since $u_0f\in L^2$ and $xu_0f\in L^2$ by assumption.
Hence $[T_{u_0},y](\ee^{\ii\lambda\diamond}a(\diamond,\lambda))$ is a constant function of $y$. More precisely,
\[
[T_{u_0},y](\ee^{\ii\lambda\diamond}a(\diamond,\lambda))
	=\frac{I_+(\Pi (u_0 \ee^{\ii\lambda\diamond} a(\cdot,\lambda)))}{2\ii\pi}
	=\frac{\overline{\widetilde{\Pi u_0}}(\lambda)}{2\ii\pi}.
\]
 Finally, by integration by parts, we can show that for every $x\in\R$ and $\lambda\in\Lambda$ 
\[
\lim_{\delta\to 0}
	 \int_{-\infty}^x\ee^{-\ii \lambda y}\frac{\dd y}{(1-\ii\delta y)^2}
	 =\ii\frac{\ee^{-\ii\lambda x}}{\lambda}.
\]

In the second term in the right-hand side in the expression of $\partial_\lambda a$, we recall  that~\cite[Lemma 32]{BlackstoneGGM24a} 
 \[
K_{u_0,\lambda}m(x):=\int_{-\infty}^x \ee^{\ii\lambda(x-y)}T_{u_0} m(y)\dd y
\]
defines a compact operator $K_{u_0,\lambda}:L^{\infty}_+\to L^{\infty}_+$. Hence 
\[
\widetilde{K}_{u_0,\lambda}b(x):= \int_{-\infty}^x \ee^{-\ii\lambda y}T_{u_0} (\ee^{\ii\lambda\diamond}b)(y)\dd y
\]
also defines a compact operator $\widetilde{K}_{u_0}:L^{\infty}_+\to L^{\infty}_+$. Therefore we can pass to the limit $\delta\to 0$ in~\eqref{eq:a-derivative} and find that $\partial_\lambda a$ is solution to 
\begin{equation}\label{a-derivative}
(\mathrm{Id}-\ii \widetilde{K}_{u_0,\lambda})\partial_\lambda a
	=-\frac{1}{2\pi}\frac{\ee^{-\ii\lambda x}}{\lambda}\overline{\widetilde{\Pi u_0}}(\lambda),
\end{equation}
where the right-hand side belongs to $L^{\infty}_+$ and $\mathrm{Id}-\ii\widetilde{K}_{u_0,\lambda}$ is an invertible operator $L^{\infty}_+\to L^{\infty}_+$. The decay of the right-hand side as $\lambda\to\infty$ is then enough to infer $\|\partial_\lambda a\|_{L^{\infty}_x}\in L^2_\lambda((\eps,\infty))$.
\end{proof}

\subsection{Asymptotics of $a$ and $\partial_\lambda a$.}
 At this stage we are going to make use of assumption \eqref{hypu}.  
 \begin{lemma}\label{lem:a-structure} 
 The functions $a$ and $\partial_\lambda a$ are bounded on $\R \times [\lambda_0,\infty )$ for every $\lambda_0>0$. Furthermore, as $y\to -\infty$, we have
 \begin{equation}\label{a-1,decay}
 a(y,\lambda)-1=-\ii \frac{c_0}{y+\ii}-\ii \frac{\overline{\widetilde {\Pi u_0}}(\lambda )}{2\pi \lambda}\frac{\ee ^{-\ii \lambda y}}{(y+\ii )}+O\left (\frac{1}{|y|^{3/2}}\right ),
 \end{equation} 
and 
\begin{equation}\label{a-derivative,decay}
 \partial_\lambda a(y,\lambda)=-\frac{\overline{\widetilde{\Pi u_0}}(\lambda)}{2\pi\lambda}\ee^{-\ii\lambda y}-\frac{\int_\R u_0(s)n(s,\lambda)\dd s}{2\pi \lambda}\frac{\ee ^{-\ii \lambda y}}{y+\ii }+O\left (\frac 1{|y|^{3/2}}\right ). 
 \end{equation}
 Moreover, the remainder terms are uniform for $(y,\lambda)\in(-\infty,-y_0]\times[\lambda_0,\infty)$.
 \end{lemma}
 \begin{proof} The boundedness of $a$ and $\partial_\lambda a$ have already been established in the proof of Lemma~\ref{lem:distortedH1}. Let us prove \eqref{a-1,decay} and \eqref{a-derivative,decay}.
 From \eqref{hypu} and  the properties~\eqref{eq:X*} of $X^*$, we infer
 \begin{equation}\label{Tum}
 T_{u_0}m_-(x,\lambda)=c_0\frac{m_-(x,\lambda)}{(x+\ii)^2}+ \frac{\ii}{2\pi}\frac{\overline{\widetilde {\Pi u_0}}(\lambda )}{x+\ii}+\frac{ c_1(\lambda)}{(x+\ii)^2}+\frac{w_1(x,\lambda)}{(x+\ii)^2},
 \end{equation}
 where \[c_1(\lambda):=\frac{\ii}{2\pi}\int_\R \frac{v_0(x)+(2\ii x-1)u_0(x)}{x+\ii}m_-(x,\lambda)\dd x\] and $w_1(x,\lambda):=T_{v_0+(2\ii x-1)u_0}m_-(x,\lambda)$. Notice that, if $\lambda \geq \lambda_0>0$, $c_1(\lambda)$ is uniformly bounded  and $w_1(\diamond,\lambda)$ is uniformly bounded in $L^2$. Next we recall that $a$ is solution to~\eqref{eq:a-1}
 so that, from~\eqref{Tum}, we obtain, as $y\to -\infty$, that $a(y,\lambda ) -1$ is bounded in $L^2$ and finally \eqref{a-1,decay} follows.\\
 We proceed similarly for the estimate of $\partial_\lambda a(y,\lambda)$ as $y\to -\infty$.  Indeed, using the decomposition \eqref{a-derivative}, we first obtain
 \begin{equation}\label{a-derivative-structure}
  \partial_\lambda a(y,\lambda)=-\frac{1}{2\pi}\frac{\ee^{-\ii\lambda y}}{\lambda}\overline{\widetilde{\Pi u_0}}(\lambda)+\ii \int_{-\infty}^y \ee^{\ii \lambda x}T_{u_0}n(x,\lambda)\dd x,
  \end{equation}
 where $n(x,\lambda):=\ee ^{\ii \lambda x }\partial_\lambda a(x, \lambda)=\partial_\lambda m_-(x,\lambda)-\ii xm_-(x,\lambda)$. 
 Then we write, similarly to \eqref{Tum},
 \begin{equation}\label{Tun}
 T_{u_0}n(x,\lambda)=c_0\frac{n(x,\lambda)}{(x+\ii)^2}+ 
  \frac{\ii}{2\pi} \frac{ \int_\R u_0(s)n(s,\lambda)\dd s}{x+\ii}+\frac{ c_2(\lambda)}{(x+\ii)^2}+\frac{w_2(x,\lambda)}{(x+\ii)^2},\end{equation}
 where
 \[c_2(\lambda):=\frac{\ii}{2\pi}\int_\R \frac{v_0(x)+(2\ii x-1)u_0(x)}{x+\ii}n(x,\lambda)\dd x\] and $w_2(x,\lambda):=T_{v_0+(2\ii x-1)u_0}n(x,\lambda)$, hence $w_2(\diamond, \lambda)\in L^2(\R)$ uniformly in $\lambda \geq \lambda_0$. 
 Since we know that $\partial_\lambda a(y,\lambda )$ is uniformly bounded, we conclude that, as $\lambda \geq \lambda_0>0$, $c_2(\lambda)$ is bounded and 
 $w_2(\diamond,\lambda)$ is uniformly bounded in $L^2$. Coming back to \eqref{a-derivative-structure}, we obtain \eqref{a-derivative,decay}.
\end{proof}

\subsection{Decaying part in $\partial_\lambda a$}
Let us define
\begin{equation}
 \label{eq:asharp} a^\sharp(y,\lambda)
 	:=\partial_\lambda a(y,\lambda)+\frac{\ee^{-\ii \lambda y}\overline{\widetilde{\Pi u_0}}(\lambda)}{2\pi \lambda}= \ii\int_{-\infty}^y \ee^{-\ii \lambda x}T_{u_0}n(x,\lambda)\dd x,
\end{equation}
where we recall that $n(x,\lambda)=\ee^{\ii \lambda x} \partial_\lambda a(x,\lambda)$.
We introduce the space $\mathscr F$ of functions $f\in L^1_{\mathrm{loc}}(\R)$ such that 
\[f(x)=f_1(x)+f_2(x),\]
where $f_2\in L^1(\R)$ and $(1+|x|)f_1(x)$ is bounded on $\R$, equipped with the norm \[\Vert f\Vert_\mathscr F:=\inf_{f=f_1+f_2}\Vert (1+|\diamond|)f_1\Vert_{L^\infty (\R)}+\Vert f_2\Vert_{L^1(\R)}.\]
\begin{lemma}\label{lem:asharp-derivative}
For $\lambda \geq \lambda_0>0$, we have
\[ \Vert \partial_ya^\sharp (\diamond ,\lambda)\Vert_{\mathscr F}+\Vert \partial_\lambda a^\sharp (\diamond,\lambda)\Vert_{L^\infty(\R)}\leq C.\]
\end{lemma}
\begin{proof} We have
\[ \partial_ya^\sharp (y,\lambda)=\ii \ee^{-\ii \lambda y}T_{u_0}n(y,\lambda),\]
and the bound follows from \eqref{Tun}. As for the second estimate, we proceed as in the proof of  \eqref{a-derivative}. We obtain
\begin{align*}
 (\mathrm{Id}-\ii \widetilde{K}_{u_0,\lambda})\partial_\lambda a^\sharp 
 	=&- \ee^{-\ii\lambda x}\frac{\int_\R u_0(s)n(s,\lambda)\dd s}{2\pi \lambda}
 	-\ii \frac{\dd }{\dd\lambda}\left(\frac{\overline{\widetilde{\Pi u_0}}(\lambda)}{2\pi \lambda}\right)\int_{-\infty}^y \ee^{-\ii \lambda x}\Pi u_0(x)\dd x\\
 	&-\frac{\overline{\widetilde{\Pi u_0}}(\lambda)}{2\pi \lambda}\int_{-\infty}^y \ee^{-\ii \lambda x}\Pi(\diamond\, u_0)(x)\dd x.
\end{align*}
We recall that  the decomposition~\eqref{hypu} and the properties~\eqref{eq:X*} of $X^*$ imply
\[
\Pi (\diamond\, u_0)(x)
=\frac{\Pi ((\diamond+\ii)\diamond u_0)(x)}{x+\ii}
	+\frac{\ii \int_{\R}su_0(s)\dd s}{2\pi(x+\ii)}
	=\frac{c_0 +\Pi (v_0)+\ii\Pi (\diamond\, u_0)(x)}{x+\ii}
	+\frac{\ii \int_{\R}su_0(s)\dd s}{2\pi(x+\ii)},
\]
so that $\int_{-\infty}^y \ee^{-\ii \lambda x}\Pi(\diamond\, u_0)(x)\dd x$ is uniformly bounded in the variable $y$.
The invertibility of $\mathrm{Id}-\ii \widetilde{K}_{u_0,\lambda}$ on $L^\infty (\R)$ combined with the decomposition \eqref{Piu, decompose} leads to the $L^\infty $ bound for $\partial_\lambda a^\sharp$.
\end{proof}

\bibliographystyle{alpha}
\bibliography{references}

@Article{BlackstoneGGM24a,
  	author = {Blackstone, E. and Gassot, L. and G\'erard, P. and Miller, P. D.},
	title = {The {B}enjamin-{O}no Initial-Value Problem for Rational Data with Application to Long Time Asymptotics and Scattering},
	JOURNAL={{Ann. Inst. H. {P}oincaré C, Anal. Non linéaire}},
    year={2025},
    doi={10.4171/AIHPC/169}
    }

@article{Gerard22,
	author = {Gérard, P.},
	doi = {10.2140/tunis.2023.5.593},
	fjournal = {Tunisian Journal of Mathematics},
	issn = {2576-7658,2576-7666},
	journal = {Tunis. J. Math.},
	mrclass = {35Q53 (35C05 37K15 47B35)},
	mrnumber = {4662323},
	number = {3},
	pages = {593--603},
	title = {An explicit formula for the {B}enjamin-{O}no equation},
	volume = {5},
	year = {2023}
}

@article{Gerard26,
title={{L}ectures on integrable equations of {B}enjamin-{O}no type},
author={G{\'e}rard, P.},
journal={EMS Surv. Math. Sci.},
year={2026},
doi={10.4171/EMSS/111}
}

@article {GeLe-24,
    AUTHOR = {G\'erard, P. and Lenzmann, E.},
     TITLE = {The {C}alogero-{M}oser derivative nonlinear {S}chr\"odinger
              equation},
   JOURNAL = {Comm. Pure Appl. Math.},
  FJOURNAL = {Communications on Pure and Applied Mathematics},
    VOLUME = {77},
      YEAR = {2024},
    NUMBER = {10},
     PAGES = {4008--4062},
      ISSN = {0010-3640,1097-0312},
   MRCLASS = {35Q55 (35C08 37K15 47B35)},
  MRNUMBER = {4814915},
}

@article {KLV25,
    AUTHOR = {Killip, R. and Laurens, T. and Vi\c{s}an, M.},
     TITLE = {Scaling-critical well-posedness for continuum
              {C}alogero-{M}oser models on the line},
   JOURNAL = {Commun. Am. Math. Soc.},
  FJOURNAL = {Communications of the American Mathematical Society},
    VOLUME = {5},
      YEAR = {2025},
     PAGES = {284--320},
      ISSN = {2692-3688},
   MRCLASS = {35Q51 (37K10)},
  MRNUMBER = {4922705},
MRREVIEWER = {Junyi\ Zhu},
       DOI = {10.1090/cams/48},
       URL = {https://doi-org.ezproxy.universite-paris-saclay.fr/10.1090/cams/48},
}

@article{KimK24,
	author = {Kim, T. and Kwon, S.},
	title = {{S}oliton resolution for {C}alogero-{M}oser derivative nonlinear {S}chr\"odinger equation},
	year = {2024},
	journal={Arxiv:2408.12843},
    archivePrefix={arXiv},
    primaryClass={math.AP},
    url = {https://arxiv.org/abs/2408.12843},
}

@article{BadreddineKV25,
author={Badreddine, R. and Killip, R. and Vi\c{s}an, M.},
title={{O}rbital stability of {B}enjamin-{O}no multisolitons},
year={2025},
journal={Arxiv:2509.14153},
 archivePrefix={arXiv},
 primaryClass={math.AP},
url = {https://arXiv.org/abs/2509.14153},
}

@book {KleinS21,
    AUTHOR = {Klein, C. and Saut, J.-C.},
     TITLE = {Nonlinear dispersive equations---inverse scattering and {PDE}
              methods},
    SERIES = {Applied Mathematical Sciences},
    VOLUME = {209},
 PUBLISHER = {Springer, Cham},
      YEAR = {[2021] \copyright 2021},
     PAGES = {xx+580},
      ISBN = {978-3-030-91426-4; 978-3-030-91427-1},
   MRCLASS = {35-02 (35C08 35Q51 35Q53 35Q55 37K40)},
  MRNUMBER = {4400881},
       DOI = {10.1007/978-3-030-91427-1},
       URL = {https://doi-org.ezproxy.universite-paris-saclay.fr/10.1007/978-3-030-91427-1},
}

@article{Chen25,
author = {Chen, X.},
title = {{S}cattering of the defocusing {C}alogero-{M}oser derivative nonlinear {S}chr\"odinger equation},
year = {2025},
journal = {Arxiv:2511.06432},
 archivePrefix={arXiv},
 primaryClass={math.AP},
url = {https://arXiv.org/abs/2511.06432},
}

@article {Wu16a,
    AUTHOR = {Wu, Y.},
     TITLE = {Simplicity and finiteness of discrete spectrum of the
              {B}enjamin-{O}no scattering operator},
   JOURNAL = {SIAM J. Math. Anal.},
  FJOURNAL = {SIAM Journal on Mathematical Analysis},
    VOLUME = {48},
      YEAR = {2016},
    NUMBER = {2},
     PAGES = {1348--1367},
      ISSN = {0036-1410,1095-7154},
   MRCLASS = {35P25 (35Q53 35R30 47A40)},
  MRNUMBER = {3484397},
       DOI = {10.1137/15M1030649},
       URL = {https://doi.org/10.1137/15M1030649},
}

@article {KaupM98,
author = {Kaup, D. J. and Matsuno, Y.},
title = {The inverse scattering transform for the {B}enjamin-{O}no equation},
journal = {Stud. Appl. Math.},
fjournal = {Studies in Applied Mathematics},
volume = {101},
pages = {73--98},
year = {1998}
}

@article {BorgheseJMcL18,
    AUTHOR = {Borghese, M. and Jenkins, R. and McLaughlin, K.              D. T.-R.},
     TITLE = {Long time asymptotic behavior of the focusing nonlinear
              {S}chr\"odinger equation},
   JOURNAL = {Ann. Inst. H. Poincar\'e{} C Anal. Non Lin\'eaire},
  FJOURNAL = {Annales de l'Institut Henri Poincar\'e{} C. Analyse Non
              Lin\'eaire},
    VOLUME = {35},
      YEAR = {2018},
    NUMBER = {4},
     PAGES = {887--920},
      ISSN = {0294-1449,1873-1430},
   MRCLASS = {35Q55 (35B40 35Q15)},
  MRNUMBER = {3795020},
MRREVIEWER = {Tatsuya\ Watanabe},
       DOI = {10.1016/j.anihpc.2017.08.006},
       URL = {https://doi-org.ezproxy.universite-paris-saclay.fr/10.1016/j.anihpc.2017.08.006},
}

@article {ChenLiu21,
    AUTHOR = {Chen, G. and Liu, J.},
     TITLE = {Soliton resolution for the focusing modified {K}d{V} equation},
   JOURNAL = {Ann. Inst. H. Poincar\'e{} C Anal. Non Lin\'eaire},
  FJOURNAL = {Annales de l'Institut Henri Poincar\'e{} C. Analyse Non
              Lin\'eaire},
    VOLUME = {38},
      YEAR = {2021},
    NUMBER = {6},
     PAGES = {2005--2071},
      ISSN = {0294-1449,1873-1430},
   MRCLASS = {35Q53 (35C08)},
  MRNUMBER = {4327906},
       DOI = {10.1016/j.anihpc.2021.02.008},
       URL = {https://doi-org.ezproxy.universite-paris-saclay.fr/10.1016/j.anihpc.2021.02.008},
}

@article {CuccJenk16,
    AUTHOR = {Cuccagna, S. and Jenkins, R.},
     TITLE = {On the asymptotic stability of {$N$}-soliton solutions of the
              defocusing nonlinear {S}chr\"odinger equation},
   JOURNAL = {Comm. Math. Phys.},
  FJOURNAL = {Communications in Mathematical Physics},
    VOLUME = {343},
      YEAR = {2016},
    NUMBER = {3},
     PAGES = {921--969},
      ISSN = {0010-3616,1432-0916},
   MRCLASS = {35Q55 (35B35 35C08)},
  MRNUMBER = {3488549},
MRREVIEWER = {Mahendra\ Panthee},
       DOI = {10.1007/s00220-016-2617-8},
       URL = {https://doi-org.ezproxy.universite-paris-saclay.fr/10.1007/s00220-016-2617-8},
}

@article {DeiftVZ94,
author={Deift, P. A. and Venakides, S. and Zhou, X.},
title = {The collisionless shock region for the long-time behavior of solutions of the {K}d{V} equation}, 
journal={Comm. Pure Appl. Math.},
fjournal = {Communications on Pure and Applied Mathematics},
volume={47},
pages={199--206},
year={1994}
}

@article {DuyKM23,
    AUTHOR = {Duyckaerts, T. and Kenig, C. and Merle, F.},
     TITLE = {Soliton resolution for the radial critical wave equation in
              all odd space dimensions},
   JOURNAL = {Acta Math.},
  FJOURNAL = {Acta Mathematica},
    VOLUME = {230},
      YEAR = {2023},
    NUMBER = {1},
     PAGES = {1--92},
      ISSN = {0001-5962,1871-2509},
   MRCLASS = {35Q35 (35R11)},
  MRNUMBER = {4567713},
       DOI = {10.4310/acta.2023.v230.n1.a1},
       URL = {https://doi-org.ezproxy.universite-paris-saclay.fr/10.4310/acta.2023.v230.n1.a1},
}

@article {DuyKM22,
    AUTHOR = {Duyckaerts, T. and Kenig, C. and Martel, Y. and
              Merle, F.},
     TITLE = {Soliton resolution for critical co-rotational wave maps and
              radial cubic wave equation},
   JOURNAL = {Comm. Math. Phys.},
  FJOURNAL = {Communications in Mathematical Physics},
    VOLUME = {391},
      YEAR = {2022},
    NUMBER = {2},
     PAGES = {779--871},
      ISSN = {0010-3616,1432-0916},
   MRCLASS = {35L05 (35C08)},
  MRNUMBER = {4397184},
       DOI = {10.1007/s00220-022-04330-z},
       URL = {https://doi-org.ezproxy.universite-paris-saclay.fr/10.1007/s00220-022-04330-z},
}

@article {EckSch83,
    AUTHOR = {Eckhaus, W. and Schuur, P.},
     TITLE = {The emergence of solitons of the {K}orteweg-de {V}ries
              equation from arbitrary initial conditions},
   JOURNAL = {Math. Methods Appl. Sci.},
  FJOURNAL = {Mathematical Methods in the Applied Sciences},
    VOLUME = {5},
      YEAR = {1983},
    NUMBER = {1},
     PAGES = {97--116},
      ISSN = {0170-4214,1099-1476},
   MRCLASS = {35Q20},
  MRNUMBER = {690898},
       DOI = {10.1002/mma.1670050108},
       URL = {https://doi-org.ezproxy.universite-paris-saclay.fr/10.1002/mma.1670050108},
}

@article {JenLaw23,
    AUTHOR = {Jendrej, J. and Lawrie, A.},
     TITLE = {Soliton resolution for the energy-critical nonlinear wave
              equation in the radial case},
   JOURNAL = {Ann. PDE},
  FJOURNAL = {Annals of PDE. Journal Dedicated to the Analysis of Problems
              from Physical Sciences},
    VOLUME = {9},
      YEAR = {2023},
    NUMBER = {2},
     PAGES = {Paper No. 18, 117},
      ISSN = {2524-5317,2199-2576},
   MRCLASS = {35L71 (35B40 37K40)},
  MRNUMBER = {4650926},
MRREVIEWER = {Enrique\ Zuazua},
       DOI = {10.1007/s40818-023-00159-4},
       URL = {https://doi-org.ezproxy.universite-paris-saclay.fr/10.1007/s40818-023-00159-4},
}

@article {JenLaw25,
    AUTHOR = {Jendrej, J. and Lawrie, A.},
     TITLE = {Soliton resolution for energy-critical wave maps in the
              equivariant case},
   JOURNAL = {J. Amer. Math. Soc.},
  FJOURNAL = {Journal of the American Mathematical Society},
    VOLUME = {38},
      YEAR = {2025},
    NUMBER = {3},
     PAGES = {783--875},
      ISSN = {0894-0347,1088-6834},
   MRCLASS = {35L71 (35B40 35C08 37K40)},
  MRNUMBER = {4890658},
       DOI = {10.1090/jams/1012},
       URL = {https://doi-org.ezproxy.universite-paris-saclay.fr/10.1090/jams/1012},
}

@article {JenkLPS18,
    AUTHOR = {Jenkins, R. and Liu, J. and Perry, P. and Sulem,
              C.},
     TITLE = {Soliton resolution for the derivative nonlinear
              {S}chr\"odinger equation},
   JOURNAL = {Comm. Math. Phys.},
  FJOURNAL = {Communications in Mathematical Physics},
    VOLUME = {363},
      YEAR = {2018},
    NUMBER = {3},
     PAGES = {1003--1049},
      ISSN = {0010-3616,1432-0916},
   MRCLASS = {35Q55 (35C08)},
  MRNUMBER = {3858827},
MRREVIEWER = {Juan\ Carlos\ Mu\~noz Grajales},
       DOI = {10.1007/s00220-018-3138-4},
       URL = {https://doi-org.ezproxy.universite-paris-saclay.fr/10.1007/s00220-018-3138-4},
}

@article {MolinetP12,
    AUTHOR = {Molinet, L. and Pilod, D.},
     TITLE = {The {C}auchy problem for the {B}enjamin-{O}no equation in
              {$L^2$} revisited},
   JOURNAL = {Anal. PDE},
  FJOURNAL = {Analysis \& PDE},
    VOLUME = {5},
      YEAR = {2012},
    NUMBER = {2},
     PAGES = {365--395},
      ISSN = {2157-5045,1948-206X},
   MRCLASS = {35Q53 (35B30)},
  MRNUMBER = {2970711},
MRREVIEWER = {John\ Albert},
       DOI = {10.2140/apde.2012.5.365},
       URL = {https://doi-org.ezproxy.universite-paris-saclay.fr/10.2140/apde.2012.5.365},
}

@article {Paulsen24,
    AUTHOR = {Paulsen, M. O.},
     TITLE = {Justification of the {B}enjamin-{O}no equation as an internal
              water waves model},
   JOURNAL = {Ann. PDE},
  FJOURNAL = {Annals of PDE. Journal Dedicated to the Analysis of Problems
              from Physical Sciences},
    VOLUME = {10},
      YEAR = {2024},
    NUMBER = {2},
     PAGES = {Paper No. 25, 129},
      ISSN = {2524-5317,2199-2576},
   MRCLASS = {35Q35 (76B55)},
  MRNUMBER = {4832425},
MRREVIEWER = {Qixiang\ Li},
       DOI = {10.1007/s40818-024-00190-z},
       URL = {https://doi-org.ezproxy.universite-paris-saclay.fr/10.1007/s40818-024-00190-z},
}

@book {Stein93,
    AUTHOR = {Stein, E. M.},
     TITLE = {Harmonic analysis: real-variable methods, orthogonality, and
              oscillatory integrals},
    SERIES = {Princeton Mathematical Series},
    VOLUME = {43},
      NOTE = {With the assistance of Timothy S. Murphy,
              Monographs in Harmonic Analysis, III},
 PUBLISHER = {Princeton University Press, Princeton, NJ},
      YEAR = {1993},
     PAGES = {xiv+695},
      ISBN = {0-691-03216-5},
   MRCLASS = {42-02 (35Sxx 43-02 47G30)},
  MRNUMBER = {1232192},
MRREVIEWER = {Michael\ Cowling},
}

@Misc{amsmath,
  author =	 {{American Mathematical Society}},
  title =	 {User's Guide for the \texttt{amsmath} Package
                  (Version 2.0)},
  url =		 {ftp://ftp.ams.org/pub/tex/doc/amsmath/amsldoc.pdf},
  urldate =	 {2015-07-30},
  year =	 2002}

@article{AmickToland1991,
	author = {Amick, C. J. and Toland, J. F.},
	doi = {10.1007/BF02392447},
	journal = {Acta Mathematica},
	number = {1},
	pages = {107--126},
	publisher = {Springer},
	title = {{Uniqueness and related analytic properties for the Benjamin-Ono equation---a nonlinear Neumann problem in the plane}},
	volume = {167},
	year = {1991}
}

@article{Benjamin67,
	address = {40 WEST 20TH STREET, NEW YORK, NY 10011-4211},
	author = {Benjamin, T. B.},
	da = {2024-05-17},
	doc-delivery-number = {99234},
	doi = {10.1017/S002211206700103X},
	issn = {0022-1120},
	journal = {J. Fluid Mech.},
	journal-iso = {J. Fluid Mech.},
	language = {English},
	number = {3},
	number-of-cited-references = {20},
	pages = {559-592},
	publisher = {CAMBRIDGE UNIV PRESS},
	research-areas = {Mechanics; Physics},
	times-cited = {843},
	title = {INTERNAL WAVES OF PERMANENT FORM IN FLUIDS OF GREAT DEPTH},
	type = {Article},
	unique-id = {WOS:A19679923400010},
	usage-count-last-180-days = {3},
	usage-count-since-2013 = {34},
	volume = {29},
	web-of-science-categories = {Mechanics; Physics, Fluids \& Plasmas},
	web-of-science-index = {Science Citation Index Expanded (SCI-EXPANDED)},
	year = {1967}
}

@article {KillipLaurensVisan23-BO,
    AUTHOR = {Killip, R. and Laurens, T. and Vi\c{s}an, M.},
     TITLE = {Sharp well-posedness for the {B}enjamin-{O}no equation},
   JOURNAL = {Invent. Math.},
  FJOURNAL = {Inventiones Mathematicae},
    VOLUME = {236},
      YEAR = {2024},
    NUMBER = {3},
     PAGES = {999--1054},
      ISSN = {0020-9910,1432-1297},
   MRCLASS = {35C08},
  MRNUMBER = {4743514},
       URL = {https://doi.org/10.1007/s00222-024-01250-8},
}

@article{IfrimTataru2019,
	author = {Ifrim, M. and Tataru, D.},
	doi = {10.24033/asens.2388},
	journal = {Ann. Sci. {\'E}c. Norm. Sup{\'e}r. (4)},
	number = {2},
	pages = {297--335},
	publisher = {{Soci{\'e}t{\'e} math{\'e}matique de France}},
	title = {{Well-posedness and dispersive decay of small data solutions for the Benjamin-Ono equation}},
	volume = {52},
	year = {2019}}

@article{Saut79,
	author = {Saut, J.-C.},
	doi = {10.1016/0022-0396(79)90068-8},
	journal = {J. Math. Pures Appl.},
	pages = {21-61},
	title = {{Sur quelques g{\'e}n{\'e}ralisations de l'{\'e}quation de Korteweg-de Vries}},
	volume = {58},
	year = {1979}
}

@article{Sun2020,
 title={{Complete integrability of the Benjamin--Ono equation on the multi-soliton manifolds}},
  author={Sun, R.},
DOI={10.1007/s00220-021-03996-1},
url={https://doi.org/10.1007/s00220-021-03996-1},
volume={383},
journal={Communications in Mathematical Physics},
publisher={Springer},
year={2021},
month={4},
pages={1051--1092}
}

\end{document}